%% file: ncs_dnn_v_1_04.tex
\documentclass[10pt]{article}
\input{preamble.tex}
\usepackage{amsmath}
\usepackage{times}
\usepackage{graphicx}
\usepackage{color}
\usepackage{multirow}
\usepackage{natbib}
\usepackage{rotating}
\usepackage{latexsym}

\newcommand{\captionfonts}{\normalsize}

\makeatletter  
\long\def\@makecaption#1#2{%
  \vskip\abovecaptionskip
  \sbox\@tempboxa{{\captionfonts #1: #2}}%
  \ifdim \wd\@tempboxa >\hsize
    {\captionfonts #1: #2\par}
  \else
    \hbox to\hsize{\hfil\box\@tempboxa\hfil}%
  \fi
  \vskip\belowcaptionskip}
\makeatother   

\begin{document}
\hspace{13.9cm}1

\ \vspace{20mm}\\

\begin{center}
{\LARGE Nonconvex Sparse Regularization for Deep Neural Networks and its Optimality}
\end{center}
\ \\
{\bf Ilsang Ohn$^{\displaystyle 1}$}\\
{$^{\displaystyle 1}$Department of Applied and Computational Mathematics and Statistics, University of Notre Dame, Notre Dame, IN 46556, USA}
\vspace{2mm}
\\
{\bf  Yongdai Kim$^{\displaystyle 2}$}\\
{$^{\displaystyle 2}$Department of Statistics, Seoul National University, Seoul 08826, Republic of Korea}\\


\thispagestyle{empty}
\markboth{}{NC instructions}
\ \vspace{-0mm}\\
%
\begin{center} {\bf Abstract}  \end{center}
Recent theoretical studies proved that deep neural network (DNN) estimators obtained by minimizing empirical risk with a certain sparsity constraint can attain optimal convergence rates for regression and classification problems. However, the sparsity constraint requires to know certain properties of the true model, which are not available in practice. Moreover, computation is difficult due to the discrete nature of the sparsity constraint. In this paper, we propose a novel penalized estimation method for sparse DNNs, which resolves the aforementioned problems existing in the sparsity constraint. We establish an oracle inequality for the excess risk of the proposed sparse-penalized DNN estimator and derive convergence rates for several learning tasks. In particular, we prove that the sparse-penalized estimator can adaptively attain minimax convergence rates for various nonparametric regression problems. For computation, we develop an efficient gradient-based optimization algorithm that guarantees the monotonic reduction of the objective function.

\vspace{2mm}

\section{Introduction}

Sparse learning of deep neural networks (DNN) has received much attention in artificial intelligence and statistics. In artificial intelligence, there are a lot of evidences \citep{han2015learning, frankle2018lottery, louizos2018learning}  to support that sparse DNN can reduce the complexity of a leaned DNN significantly (in terms of the number of parameters as well as the numbers of hidden layers and hidden nodes) without hampering prediction accuracy much. By doing so, we can reduce memory and energy consumption at the prediction phase. 

In statistics, recent studies about DNNs for nonparametric regression and classification \citep{schmidt2020nonparametric, imaizumi2018deep, suzuki2018adaptivity, bauer2019deep, kim2018fast} proved that a DNN estimator minimizing an empirical risk with a certain sparsity constraint achieves the minimax optimality for a wide class of functions including smooth functions, piecewise smooth functions and smooth decision boundaries. However, there are still two unanswered questions. The first question is how to choose a suitable level of sparsity, which depends on the unknown smoothness and/or the unknown intrinsic dimensionality of the true function. The second question is computation. Learning a deep architecture with a given sparsity constraint is computationally intractable since we need to explore a large number of possible configurations of sparsity pattern in the network parameter.

More recently, \citet{kohler2019rate} showed that the empirical risk minimizer over fully connected (i.e., non-sparse) DNNs can have minimax optimality also. Although removing the sparse constraint circumvents the related computation issue, their result is still nonadaptive because the appropriate number of hidden nodes should depend on the unknown smoothness of the true regression function.

In this paper, we propose a novel learning method of sparse DNNs for nonparametric regression and classification, which answers the aforementioned two questions in the sparsity-constrained empirical risk minimization (ERM) method. The proposed learning algorithm is to learn a DNN by minimizing the penalized empirical risk, which is the sum of the empirical risk and the clipped $L_1$ penalty \citep{zhang2010analysis}. By choosing the position of the clipping in the clipped $L_1$ penalty carefully, we establish an oracle inequality for the excess risk of the proposed sparse DNN estimator and derive convergence rates for several learning tasks. In particular, it will be shown that the proposed DNN estimator can \textit{adaptively} attain minimax convergence rates for various nonparametric regression problems.   

Although nonconvex penalties such as the clipped $L_1$ penalty are popular for high-dimensional linear regressions \citep{fan2001variable, zhang2010nearly}, they are not popularly used for DNN. Instead, $L_1$ norm-based penalties such as Lasso and Group Lasso are popular \citep{liu2015sparse, wen2016learning}. This would be partly because of the convexity of the $L_1$  penalty.  For computation with the clipped $L_1$ penalty, we develop an mini-batch optimization algorithm by combining  the proximal gradient descent algorithm \citep{parikh2014proximal} and the concave-convex procedure (CCCP) \citep{yuille2003concave}. The CCCP is a procedure to replace the clipped $L_1$ penalty by its tight convex upper bound to make the optimization problem be $L_1$ penalized, and the proximal gradient descent algorithm is a mini-batch optimization algorithm for $L_1$ penalized optimization problems. 
 
\subsection{Notation}

We denote by $\ind(\cdot)$ the indicator function. Let $\R$ be the set of real numbers and $\bN$ be the set of natural numbers. Let $\bN_0:=\bN\cup\{0\}.$ Let $[m]:=\{1,2,\dots,m\}$ for $m\in\bN$. For two real numbers $a$ and $b$, we write $a\vee b:=\max\{a,b\}$ and $a\wedge b:=\min\{a,b\}$. For a real valued vector $\x\equiv(x_1,\dots, x_d)\in\R^d$, we let $\norm{\x}_0:=\sum_{j=1}^d\mathbbm{1}(x_j\neq 0)$, $\norm{\x}_p:=\del{\sum_{j=1}^d|x_j|^p}^{1/p}$ for $p\in[1,\infty)$ and $\norm{\x}_\infty:=\max_{1\le j\le d}|x_j|$. For a real-valued function $f:\cX\mapsto \R$, we let $\norm{f}_{\infty, \cX}:= \sup_{\x\in\cX}|f(\x)|.$ If the domain of the function $f$ is clear in the context, we omit the subscript $\cX$ to write  $\norm{f}_{\infty}:= \norm{f}_{\infty, \cX}$. For $p\in[1,\infty)$ and a distribution $\sQ$ on $\cX$, let $\|f\|_{p,Q}:=\del{\int|f(\x)|^p\d\sQ(\x)}^{1/p}$. For two positive sequences $\{a_n\}_{n\in \mathbb{N}}$ and $\{b_n\}_{n\in \mathbb{N}}$, we write $a_n\lesssim b_n$ or $b_n\gtrsim a_n$ if there exists a positive constant $C>0$ such that $a_n\le Cb_n$ for any $n\in \mathbb{N}$. We also write $a_n\asymp b_n$ if $a_n\lesssim b_n$  and  $a_n\gtrsim b_n$. 

\subsection{Deep neural networks}

A DNN with $L\in\bN$ layers, $N_l\in\bN$ many nodes at the $l$-th hidden layer for $l=[L]$, input of dimension $N_0$, output of dimension $N_{L+1}$ and nonlinear activation function $\rho:\R\mapsto\R$ is expressed as
    \begin{equation}
    \label{eq:nn}
        f(\x)=A_{L+1}\circ\rho_L\circ A_{L}\circ\cdots \circ\rho_1\circ A_1(\x),
    \end{equation}
where $A_l:\R^{N_{l-1}}\mapsto\R^{N_l}$ is an affine linear map defined by $A_l(\x)=\W_l\x+\b_l$ for given  $N_l\times N_{l-1}$ dimensional weight matrix $\W_l$ and $N_l$ dimensional bias vector $\b_l$, and $\rho_l:\R^{N_l}\mapsto\R^{N_{l}}$ is an element-wise nonlinear activation map defined as $\rho_l(\z):=(\rho(z_1),\dots, \rho(z_{N_l}))^\top$. We let $\btheta(f)$ denote a parameter, which is a concatenation of  all the weight matrices and the bias vectors, of the DNN $f$. That is,  
    $$\btheta(f):=\del[0]{\text{vec}(\W_1)^\top,\b_1^\top,\dots, \text{vec}(\W_{L+1})^\top, \b_{L+1}
^\top}^\top,$$
where $\text{vec}(\W)$ transforms the matrix $\W$ into the corresponding vector by concatenating the column vectors. 
 
We let $\cF_{\rho, d,o}^\dnn$ be the class of DNNs which take $d$-dimensional input (i.e., $N_0=d$) to produce $o$-dimensional output (i.e., $N_{L+1}=o$) and use the activation function $\rho:\R\mapsto\R$. In this paper, we focus on real-valued DNNs, i.e., $o=1,$ but the results in this paper can be extended easily for the case of $o\ge 2.$ 

For a given DNN $f$, we let $\textsf{depth}(f)$ denote  the depth (i.e., the number of hidden layers) and $\textsf{width}(f)$ denote  the width (i.e., the maximum of the numbers of hidden nodes at each layer) of the DNN $f$.  Throughout this paper, we  consider a class of DNNs with some constraints on the architecture, parameter and output value of a DNN such that
    \begin{equations}
        \cF^\dnn_\rho(L, N, B, F):=\big\{f\in\cF_{\rho, d,1}^\dnn:
        & \textsf{depth}(f)\le L, \textsf{width}(f)\le N,\\
        &\|\btheta(f)\|_\infty\le B, \|f\|_\infty \le F\big\}
    \end{equations}    
for positive constants $L$, $N$, $B$ and $F$. We consider $C$-Lipschitz $\rho$ for some $C>0.$ That is,
there exists $C>0$ such that $|\rho(z_1)-\rho(z_2)|\le C |z_1-z_2|$ for any $z_1,z_2\in\R$.
The ReLU activation function $x\mapsto \max\{0,x\}$ and the sigmoid activation function $x\mapsto 1/(1+\e^{-x})$, which are the two most popularly used activation functions, are both $C$-Lipschitz. Various $C$-Lipschitz activation functions are listed in \cref{sec:activation_examples}.

\subsection{Empirical risk minimization algorithm with sparsity constraint and its nonadaptiveness}
\label{sec:nonadap}

Most studies about DNNs for nonparametric regression \citep{bauer2019deep, suzuki2018adaptivity, imaizumi2018deep,imaizumi2020advantage,schmidt2020nonparametric, tsuji2021estimation} consider the ERM method with a certain sparsity constraint which can be summarized as follows. Let $(\X_1,Y_1),\ldots,(\X_n,Y_n)$ be $n$ many input-output pairs which are assumed to be independent random vectors identically distributed according to $\P$ on $\cX\times \cY$, where $\cX$ is a compact subset of $\R^d$ and $\cY$ is a subset of $\R$. First, a class of sparsity constrained DNNs with sparsity level $S>0$ is defined as
    \begin{equations}
        \cF^\dnn_\rho(L, N, B, F, S):=\big\{f\in\cF^\dnn _\rho(L, N, B, F):
        \|\btheta(f)\|_0\le S\big\}.
    \end{equations}   
Then for a given loss $\ell:\cY\times \R \rightarrow \R_+,$
the sparsity-constrained ERM estimator is defined as
     \begin{equation}
     \label{eq:erm_enn}
    \hat{f}_n^{\textsc{erm}}\in\argmin_{f\in \cF^\dnn_\rho(L_n, N_n, B_n, F_n, S_n)} \frac{1}{n}\sum_{i=1}^n \ell(Y_i,f(\X_i))
    \end{equation}  
with suitably chosen architecture parameters $L_n, N_n, B_n,  F_n$ and sparsity $S_n.$ 

It has been proven that the estimator $\hat{f}_n^{\textsc{erm}}$ attains minimax optimality in various supervised learning tasks, but most results are nonadaptive \citep{bauer2019deep, suzuki2018adaptivity, imaizumi2018deep,imaizumi2020advantage,schmidt2020nonparametric, tsuji2021estimation, kim2018fast}.  
To be more specific, let $f_\ell^\star:=\argmin_{f\in \cF} \E\ell(Y, f(\X))$, where $\cF$ is a set of all real-valued measurable function on $\cX$. Define the excess risk of a function $f$ as
    \begin{equation*}
        \cE_\P(f):=\E\ell(Y,f(\X))-\E\ell(Y,f_\ell^\star(\X)).
    \end{equation*}
If $\ell$ is the square loss, the activation function is the ReLU and $f_\ell^\star$ belongs to the class of H\"{o}lder functions of smootheness $\alpha>0$ with radius $R$ (see \labelcref{eq:def_holder} in Section \ref{sec:theory} for the definition of H\"{o}lder functions), \citet{schmidt2020nonparametric} proves that
the convergence rate of the excess risk $\cE_\P( \hat{f}_n^{\textsc{erm}})$
is $O(n^{-\frac{2\alpha}{2\alpha+d}} \log^3 n),$ which is is minimax optimal up to a logarithmic factor,
provided that $L_n\lesssim \log n$, $N_n\lesssim n^{\nu_1}, B_n\lesssim n^{\nu_2}$ and $S_n\asymp n^{\frac{d}{2\alpha+d}}\log n$ for some positive constants $\nu_1$ and $\nu_2.$ That is, the sparsity level $S_n$ for attaining the minimax optimality depends on the smoothness $\alpha$ of the true function $f_\ell^\star$ which is unknown. This nonadaptiveness still exists for classification. For details, see \cite{kim2018fast}.

\subsection{Outline}

The rest of the paper is organized as follows. In \cref{sec:dnn}, we propose a sparse penalized learning method for DNNs. In \cref{sec:theory}, we provide the oracle inequalities for the proposed sparse DNN estimator. Based on these oracle inequalities, we derive the convergence rates of our estimator for several supervised learning problems. 
In \cref{sec:compute}, we develop a computational algorithm.
In \cref{sec:numerical}, we conduct numerical study to assess the finite-sample performance of our estimator. Concluding remarks follow in \cref{sec:conclusion}, and the proofs are gathered in \cref{sec:sparse_proofs}. Approximation properties of DNNs with various activation functions are provided in \cref{sec:activation}.

\section{Learning sparse deep neural networks with the clipped $L_1$ penalty}
\label{sec:dnn}

In this paper, we consider the penalized empirical risk minimizer over DNNs, which is defined as
    \begin{equation}
    \label{eq:estimator}
    \hat{f}_n\in\argmin_{f\in \cF_n^\dnn}  
    \sbr{\frac{1}{n} \sum_{i=1}^n  \ell\left( Y_i, f(\X_i)\right)+J_n(f)},
    \end{equation}
where $\cF_n^\dnn$ is a certain class of DNNs and $J_n(f)$ a sparse penalty function. 
We call $\hat{f}_n$ the \textit{sparse-penalized DNN estimator}.
For the sparse penalty $J_n(f),$
we propose to use the clipped $L_1$ penalty given by
    \begin{equation}
        J_n(f):=J_{\lambda_n, \tau_n}(f):=\lambda_n\norm{\btheta(f)}_{\clip, \tau_n},
    \end{equation}
for tuning parameters $\lambda_n>0$ and $\tau_n>0,$
where $\|\cdot\|_{\clip, \tau}$ denotes the \textit{clipped $L_1$ norm} with a clipping threshold $\tau>0$ \citep{zhang2010analysis}  defined as
    \begin{equation}
        \|\btheta\|_{\clip,\tau}:=\sum_{j=1}^{p}\left( \frac{|\theta_j|}{\tau}\wedge 1\right)
    \end{equation}
for a $p$-dimensional vector $\btheta\equiv(\theta_j)_{j\in[p]}.$
    
The clipped $L_1$  norm  can be viewed as a continuous relaxation of the $L_0$ norm $\|\btheta\|_0$. \cref{fig:pen} compares the $L_0$ and the clipped $L_1$ norms. The continuity of the clipped $L_1$ norm makes the optimization \labelcref{eq:estimator} much easier than that with the $L_0$ norm, which will be discussed in \cref{sec:compute}.

The clipped $L_1$ norm has been used for sparse high dimensional linear regression by  \cite{zhang2010analysis} which yields an estimator having the oracle property. The main results of this paper is that  with suitable choices for $\lambda_n$ and $\tau_n$, which do depend on neither training data nor the true distribution, the sparse-penalized DNN estimator \labelcref{eq:estimator} with the clipped $L_1$ penalty can adaptively attain minimax optimality.

\begin{figure}
    \centering
    \includegraphics[scale=0.6]{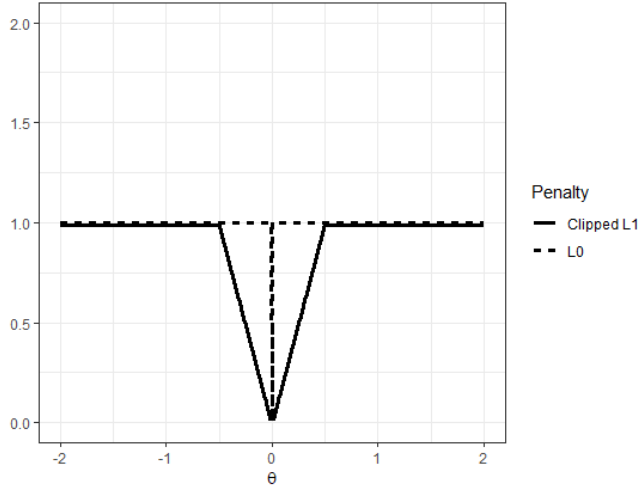}
    \caption{The clipped $L_1$ and $L_0$ penalties}
    \label{fig:pen}
\end{figure}

\section{Main results}
\label{sec:theory}

In this section, we provide theoretical justifications of the sparse-penalized DNN estimator (\ref{eq:estimator})  in both regression and binary classification tasks. We prove that minimax optimal convergence rates of the excess risk can be obtained adaptively for various nonparametric regression and classification tasks.

\subsection{Nonparametric regression}
\label{sec:nonpar_reg}

We first consider a nonparametric regression task, where the response $Y\in\R$ and input $\X\in [0,1]^d$ are generated from the model
    \begin{equation}
    \label{eq:model_reg}
        Y=f^\star(\X)+\epsilon, \quad \X\sim \P_\X,
    \end{equation}
where  $f^\star:[0,1]^d\mapsto \R$ is the unknown true regression function, $\P_\X$ is a distribution on $[0,1]^d$ and $\epsilon$ is an  error variable independent to the input variable $\X$. 
For technical simplicity, we focus on the sub-Gaussian error such that
    \begin{equation}
    \label{eq:subgauss}
        \E(\e^{t\epsilon})\le \e^{t^2\sigma^2/2}
    \end{equation}
for any $t\in\bR$ for some $\sigma>0$. We denote by $\cP_{\sigma, F^\star}$ the set of distributions $(\X,Y)$ satisfying 
the model \labelcref {eq:model_reg}:
    \begin{equation*}
        \cP_{\sigma, F^\star}:=\cbr{\mbox{Model \labelcref{eq:model_reg}}:
        \E(\e^{t\epsilon})\le\e^{t^2\sigma^2/2},\forall t\in\bR, \|f^\star\|_\infty\le F^\star}.
    \end{equation*}
The problem is to estimate the unknown true regression function $f^\star$ based on given training data $\{(\X_1,Y_i)\}_{i\in[n]}\sim \P^n$ where $\P\in\cP_{a,F^\star}$. We evaluate the performance of an estimator $\hat{f}$ by the expected $L_2(\P_\X)$ error 
    \begin{equation*}
        \E\sbr{\|\hat{f}-f^\star\|^2_{2, \P_\X}},
    \end{equation*}
where the expectation is taken over the training data and 
    \begin{equation*}
        \|\hat{f}-f^\star\|^2_{2,\P_\X}:=\int|\hat{f}(\x)-f^\star(\x)|^2\d\P_\X(\x).
    \end{equation*}

The following theorem provides an oracle inequality for the expected $L_2(\P_\X)$ error of the sparse-penalized DNN estimator.

\begin{theorem}
\label{thm:oracle_reg}
Assume that the true generative model $\P$ is in $\cP_{\sigma, F^\star}$. Let $F>0$ and let $\{L_n\}_{n\in\bN}$, $\{N_n\}_{n\in\bN}$ and $\{B_n\}_{n\in\bN}$ be positive sequences such that  $L_n\lesssim \log n$, $N_n\lesssim n^{\nu_1}$, $1\le B_n\lesssim n^{\nu_2}$ for some $\nu_1,\nu_2>0$. Then the sparse-penalized DNN estimator  defined by
    \begin{equation}
    \label{eq:estimator_reg}
        \hat{f}_n\in\argmin_{f\in\cF^\dnn_n}\sbr{\frac{1}{n}\sum_{i=1}^n(Y_i-f(\X_i))^2+\lambda_n\norm{\btheta(f)}_{\clip, \tau_n}},
    \end{equation}
with $\cF^\dnn_n=\cF^\dnn_\rho(L_n, N_n, B_n, F)$, $\lambda_n\asymp \log^5n/n$ and
$-\log \tau_n \ge A\log^{2}n$ for sufficiently large $A>0$,
satisfies
    \begin{equation}
    \label{eq:oracle_reg}
        \E\sbr{\norm[0]{\hat{f}_n-f^\star}_{2, \P_\X}^2}\le
        2\inf_{f\in\cF^\dnn_n}\cbr{ \norm{f-f^\star}_{2, \P_\X}^2+\lambda_n\norm{\btheta(f)}_{\clip, \tau_n}} \vee\frac{C_{\sigma,F^\star}\log^2n}{n}
    \end{equation}
for some constant $C_{\sigma,F^\star}>0$ depending only on $\sigma$ and $F^\star$, where the expectation is taken over the training data.
\end{theorem}

The following theorem, which is a corollary of \cref{thm:oracle_reg}, provides a useful tool to derive convergence rates of the sparse-penalized DNN estimator for various classes of functions to which the true regression function $f^\star$ belongs. 

\begin{theorem}
\label{cor:conv_reg}
Let $\{L_n\}_{n\in\bN}$, $\{N_n\}_{n\in\bN}$, $\{B_n\}_{n\in\bN}$ and $F>0$ be as in \cref{thm:oracle_reg}. Moreover, let $F^\star>0$ and let $\cF^\star$ be a set of some real-valued functions on $[0,1]^d$. Assume that there are  constants $\kappa>0$, $r>0$, $\epsilon_0>0$ and  $C>0$ such that
    \begin{equation}
    \label{eq:approx_reg}
       \sup_{f^\diamond\in\cF^\star:\|f^\diamond\|_\infty\le F^\star}\inf_{f\in\cF^\dnn_\rho(L_n, N_n, B_n, F,S_{n,\epsilon})}\|f-f^\diamond\|_{2,\P_\X}\le \epsilon
    \end{equation}
with $S_{n,\epsilon}:=C\epsilon^{-\kappa}\log^rn$ for any $\epsilon\in(0,\epsilon_0)$ and $n\in\bN.$ Then the sparse-penalized DNN estimator defined by \labelcref{eq:estimator_reg} with $\cF^\dnn_n=\cF^\dnn_\rho(L_n, N_n, B_n, F)$ satisfies
    \begin{equation}
        \sup_{\P\in\cP_{\sigma,F^\star}:f^\star\in\cF^\star}\E\sbr{\norm[0]{\hat{f}_n-f^\star}_{2, \P_\X}^2}\lesssim n^{-\frac{2}{\kappa+2}}\log^{5+r}n.
    \end{equation}
\end{theorem}

If there exist positive constants $\nu_1$ and $\nu_2$ in \cref{thm:oracle_reg} with which the condition \labelcref{eq:approx_reg} holds for wide classes of $\cF^\star$, then the convergence rate of the sparse-penalized estimator can be adaptive to the choice of $\cF^\star.$ In the followings, we list up several examples where the sparse-penalized estimator is adaptively minimax optimal (up to a logarithmic factor). Before this, we consider the two types of activation functions given below because the constant $\nu_2$ differs for these two types of activation functions.

\begin{definition}
A  function $\rho:\R\to\R$ is {\it continuous piecewise linear} if it is continuous and there exist a finite number of break points $a_1\le a_2\le \cdots\le  a_K\in\R$ with $K\in\bN$ such that  $\rho'(a_k-)\neq \rho'(a_k+)$ for every $k\in[K]$ and $\rho(x)$ is linear on $(-\infty,a_1]$, $[a_1,a_2]$, $\dots$, $[a_{K-1}, a_K]$, $[a_K, \infty)$.
\end{definition}

\begin{definition}
A function $\rho:\R\to\R$ is {\it locally quadratic} if there exits an open interval $(a,b)\subset\R$ on which $\rho$ is three times continuously differentiable with bounded derivatives and there exists $t\in(a,b)$ such that $\rho'(t)\neq0$ and $\rho''(t)\neq0$. 
\end{definition}

Examples of continuous piecewise linear and locally quadratic activation functions are ReLU and sigmoid functions, respectively. Other activation functions are listed in \cref{sec:activation_examples}.

\paragraph{H\"older functions} The H\"older space of smoothness $\alpha>0$ with radius $R>0$ is defined as 
	\begin{equation}
	\label{eq:def_holder}
	\cH^{\alpha, R}(\cX):=\left\{f:\cX\mapsto\R:\|f\|_{\cH^{\alpha}(\cX)}\le R\right\},
	\end{equation}
where $\|f\|_{\cH^{\alpha}(\cX)}$ denotes the H\"older norm defined by
    \begin{align*}
    \|f\|_{\cH^{\alpha}(\cX)}
        :=&\sum_{\m\in\bN_0^d:\|\m\|_1\le \floor{\alpha}}\|\partial^{\m}f\|_{\infty}\\
        &+\sum_{\m\in\bN_0^d:\|\m\|_1= \floor{\alpha}}\sup_{\x_1,\x_2\in \cX, \x_1\neq\x_2 }\frac{|\partial^{\m}f(\x_1)-\partial^{\m}f(\x_2)|}{|\x_1-\x_2|^{\alpha- \floor{\alpha}}}.
    \end{align*}
Here, $\partial^{\m}f$ denotes the partial derivative of $f$ of order $\m$.

\citet{yarotsky2017error} and \citet{schmidt2020nonparametric} proved that for $\cH^{\alpha, R}([0,1]^d)$, the class of DNNs $\cF^\dnn_\rho(L_n, N_n, B_n, F, S)$ with the ReLU activation $\rho$ and 
    \begin{equation*}
        L_n\asymp \log n, \quad N_n\gtrsim n^{\frac{d}{2\alpha+d}}, \quad B_n=1, \quad F>F^\star
    \end{equation*}
satisfies the condition \labelcref{eq:approx_reg} with  $\kappa = d/\alpha$ and $r=1$ for all $\alpha>0.$ Hence, for $L_n\asymp \log n$, $N_n\asymp n$, $B_n=B\ge1$, $F>F^\star$ and the ReLU activation function $\rho$, \cref{cor:conv_reg} implies that the convergence rate of the sparse-penalized DNN estimator defined by \labelcref{eq:estimator_reg} with $\cF^\dnn_n=\cF^\dnn_\rho(L_n, N_n, B_n, F)$ is given by
    \begin{equation}
    \label{eq:rate_holder}
        n^{-\frac{2\alpha}{2\alpha+d}}\log^6n,
    \end{equation}
which is the minimax optimal (up to a logarithmic factor). That is, the sparse-penalized DNN estimator is minimax-optimal adaptively to the smootheness $\alpha.$

Also Theorem 1 of \citet{ohn2019smooth} (which is presented in \cref{thm:holder_approx} in \cref{sec:activation} for reader's convenience) shows that similar approximation results hold for piecewise linear activation function with $B_n\asymp1$ and locally quadratic activation functions   
with $B_n\asymp n^4.$ 

\paragraph{Composition structured functions}

The curse of dimensionality can be avoided by certain structural assumptions on the regression function. \citet{schmidt2020nonparametric} considered so-called composition structured regression functions which include a single-index model \citep{gaiffas2007optimal}, an additive model \citep{stone1985additive, buja1989linear} and a generalized additive model with an unknown link function \citep{horowitz2007rate} as special cases. This class is specified as follows. 
Let $q\in\bN$, $\mathbf{d}:=(d_1, \dots, d_{q+1})\in\bN^{q+1}$ with $d_1:=d$ and $d_{q+1}=1$, $\t:=( t†_1, \dots, t_{q+1})\in\prod_{j=1}^{q+1}[d_j]$, and $\balpha:=(\alpha_1, \dots, \alpha_q)\in\R^{q}$.
We denote by $\cG^{\textsc{comp}}(q, \balpha, \mathbf{d}, \t, R)$ a set of composition structured function given by 
    \begin{equations}
    \label{eq:compose}
        &\cG^{\textsc{comp}}(q, \balpha, \mathbf{d}, \t, R)\\
        &:=\big\{f=g_q\circ\cdots\circ g_1: g_{j}=(g_{j,k})_{k\in[d_{j+1}]}:[a_j,b_j]^{d_j}\mapsto[a_{j+1},b_{j+1}]^{d_{j+1}}, \\ 
        &\hspace{40mm} g_{j,k}\in \cH^{\alpha_j, R}([a_j,b_j]^{t_j})\mbox{ for some $|a_j|\vee|b_j|\le R$}\big\}.
    \end{equations}

Letting $\alpha^*_j:=\alpha_j\prod_{h=j+1}^q(\alpha_h\wedge 1)$ for each $j\in[q-1]$ and $\alpha_q^*:=\alpha_q$, \citet{schmidt2020nonparametric} showed that for the function class $\cG^{\textsc{comp}}(q, \balpha, \mathbf{d}, \t, R)$ in \labelcref{eq:compose}, the class of DNNs $\cF^\dnn_\rho(L_n, N_n, B_n, F, S)$ with the ReLU activation $\rho$ and
    \begin{equation*}
        L_n\asymp \log n, \quad N_n\gtrsim \max_{j\in [q]}n^{\frac{t_j}{2\alpha^*_j+t_j}}, \quad B_n=1, \quad  F>F^\star
    \end{equation*}
satisfies the condition \labelcref{eq:approx_reg}  with
    $$\kappa =\max_{j\in [q]}\frac{t_j}{\alpha^*_j}$$
and $r=1$.
Thus for $L_n\asymp \log n$, $N_n\asymp n$, $B_n=B\ge1$, $F>F^\star$ and the ReLU activation function $\rho$, the sparse-penalized DNN estimator defined by \labelcref{eq:estimator_reg} with $\cF^\dnn_n=\cF^\dnn_\rho(L_n, N_n, B_n, F)$ attains the rate
    \begin{equation}
    \label{eq:rate_compose}
        \max_{j\in [q]}n^{-\frac{2\alpha^*_j}{2\alpha^*_j+t_j}}\log ^6n,
    \end{equation}
which is minimax optimal up to a logarithmic factor. \cref{thm:composite_approx} in \cref{sec:activation} shows that a similar approximation result holds for the  piecewise linear activation functions with $B_n\asymp 1$ and hence the corresponding sparse-penalized DNN estimator is minimax optimal adaptively to $(\alpha_j:j\in[q])$.

For locally quadratic  activation functions, a situation is tricky. In \cref{sec:activation}, we succeeded in proving only that there exists $\nu_2>0$ satisfies \cref{cor:conv_reg} only when there exists $\xi>0$ such that $\min_{j\in[q]} \alpha_j >\xi.$ That is, the sparse-penalized DNN estimator is adaptively minimax optimal only for sufficiently smooth functions. However, we think that this minor incompleteness would be mainly due to technical limitations.

\paragraph{Piecewise smooth functions}

\citet{petersen2018optimal} and \citet{imaizumi2018deep} introduced a notion of piecewise smooth functions, which have a support divided into several pieces with smooth boundaries and are smooth only within each of the pieces. Let $M\in\bN$, $K\in\bN$, $\alpha>0$, $\beta>0$ and $R>0$.  Formally, the class of piecewise smooth functions is defined as
   \begin{equations}
   \label{eq:piece}
        \cG^{\textsc{piece}}(\alpha, \beta, M, K, R) :=\big\{f:&f(\x)=\sum_{m=1}^Mg_m(\x)\prod_{k\in[K]}\mathbbm{1}\del{x_{j_{m,k}}\ge h_{m,k}(\x_{-j_{m,k}})},\\ 
        & g_m\in \cH^{\alpha, R}([0,1]^d), h_{m,k}\in \cH^{\beta, R}([0,1]^{d-1}), j_{m,k}\in[d]
        \big\}.
    \end{equations}

\citet{petersen2018optimal} and \citet{imaizumi2018deep} showed that for the function class $\cG^{\textsc{piece}}(\alpha, \beta,M, K, R)$ in \labelcref{eq:piece}, the class of DNNs $\cF^\dnn_\rho(L_n, N_n, B_n, F, S)$ with the ReLU activation function and
    \begin{equation*}
        L_n\asymp \log n, \quad  
        N_n\gtrsim n^{\frac{d}{2\alpha+d}}\vee n^{\frac{d-1}{\beta+d-1}}, \quad  
        B_n\gtrsim  n^{\frac{\alpha}{2\alpha+d}}\vee n^{\frac{\beta}{2(\beta+d-1)}}, \quad 
        F>F^\star
    \end{equation*}
satisfies the condition \labelcref{eq:approx_reg} with
    \begin{equation*}
        \kappa = \frac{d}{\alpha}\vee\frac{2(d-1)}{\beta}
    \end{equation*}
and $r=1$, provided that the marginal distribution $\P_\X$ of the input variable admits a density $\frac{\d\P_\X}{\d\mu}$ with respect to the Lebesgue measure $\mu$ and $\sup_{\x\in[0,1]^d}\frac{\d\P_\X}{\d\mu}(\x)\le C$ for some $C>0$. Hence for $L_n\asymp \log n$, $N_n\asymp n$, $B_n\asymp n$, $F>F^\star$ and the ReLU activation function $\rho$, the sparse-penalized  DNN estimator defined by \labelcref{eq:estimator_reg} with $\cF^\dnn_n=\cF^\dnn_\rho(L_n, N_n, B_n, F)$ attains the rate
    \begin{equation}
    \label{eq:rate_piece}
        \cbr{n^{-\frac{2\alpha}{2\alpha+d}}\vee n^{-\frac{\beta}{\beta+d-1}}}\log^6n,
    \end{equation}
which is minimax optimal up to a logarithmic factor.

\cref{thm:piece_approx} in  \cref{sec:activation} shows that a similar DNN approximation result holds for piecewise linear activation functions with $B_n\asymp n$ and locally quadratic activation functions with $B_n\asymp n^4$. Hence  the sparse-penalized DNN estimator with the activation function being either piecewise linear or locally quadratic is also minimax optimal adaptively to $\alpha$ and $\beta$.

\paragraph{Besov and mixed smooth Besov functions} 
\citet{suzuki2018adaptivity} proved the minimax optimality of the ERM estimator with a certain sparsity constraint for the estimation of a regression function in the Besov space or the mixed smooth Besov space. Similarly to the other function spaces, we can prove that the sparse-penalized DNN estimator is minimax optimal adaptively for the Besov space or the mixed smooth Besov space using \cref{cor:conv_reg} along with Proposition 1 and Theorem 1 of \citet{suzuki2018adaptivity}, respectively. We omit the details due to the limitation of spaces.

\paragraph{Summary of the network architecture}

In \cref{tab:architecture}, we summarize the minimal values of the network architecture parameters $\nu_1$ and $\nu_2$ that attain the adaptive optimality according to the type of  activation function and the class of true regression functions. Note that any values of the architecture parameters larger than the corresponding minimal values in the table also lead to the adaptive optimality.  Thus, users can select $\nu_1$ and $\nu_2$  based on the prior information about the true regression function and the results in the table without resorting to a tuning procedure. The choice $\nu_1=1$ is allowed regardless of the true regression function and the choice of the activation function. Any $\nu_1$ larger than 1 can be used but additional computation is needed since more hidden nodes are used.  For $\nu_2$, we may need a very large value, in particular, when we use the locally quadratic activation function and the true regression function is of composition structured. But since the boundness restriction of the parameter does not affect its computational complexity and thus we recommend to use a sufficiently large $\nu_2.$

\begin{table}[]
\centering
\caption{Summary of the minimal values of the network architecture parameters $\nu_1$ and $\nu_2$  that attain the adaptive optimality 
according to the activation function and the class of true regression functions. Here, $\xi$ denotes the lower bound of the smoothness $(\alpha_j:j\in[q]).$}
\label{tab:architecture}
\begin{tabular}{llcc}
\hline
True regression function                        & Activation function & $\nu_1$ & $\nu_2$                         \\\hline
\multirow{2}{*}{H\"older smooth} & Piecewise linear    & 1       & 0                               \\\cline{2-4}
                                                & Locally quadratic   & 1       & 4                               \\\hline
\multirow{2}{*}{Composition structured}         & Piecewise linear    & 1       & 0                               \\\cline{2-4}
                                                & Locally quadratic   & 1       & $4+\frac{1}{2}\max\{0,\xi^{-q+1}-4\}$ \\\hline
\multirow{2}{*}{Piecewise smooth}               & Piecewise linear    & 1       & 1                               \\\cline{2-4}
                                                & Locally quadratic   & 1       & 4                              \\\hline
\end{tabular}
\end{table}

\subsection{Classification with strictly convex losses}

In this section, we consider a binary classification problem. The goal of classification is to find a real-valued function $f$ (called a \textit{decision function}) such that $f(\x)$ is a good prediction of the label $y\in\{-1,1\}$ for a new sample $(\x, y)$. In practice, the margin-based loss function, which evaluates the quality of the prediction by $f$ for a sample $(\x, y)$  based on its margin $yf(\x)$, is popularly used. Examples of the margin based loss  functions are the 0-1 loss $\ind(yf(\x)<0)$, hinge loss $(1-yf(\x))\vee0$,   exponential loss $\exp(-yf(\x))$ and logistic loss $\log(1+\exp(-yf(\x)))$. Here we focus on strictly convex losses  which include the exponential and the logistic losses. Note that the logistic loss is popularly used for learning a DNN classifier in practice under the name of cross-entropy.

We assume that the label $Y\in\{-1,1\}$ and input $\X\in[0,1]^d$ are generated from the model
    \begin{equation}
    \label{eq:model_reg_cls}
        Y|\X=\x\sim 2\mathsf{Bernoulli}(\eta(\x))-1,\quad \X\sim \P_\X,
    \end{equation}
where $\eta(\x)$ is called a conditional class probability and $\P_\X$ is a distribution on $[0,1]^d$. The aim is to find a real-valued function $f$ so that the \textit{excess risk} of $f$ given by:
    \begin{equation*}
        \cE_{\P}(f):=\E(\ell(Yf(\X)))-\E(\ell(Yf_\ell^\star(\X)))
    \end{equation*}
close to zero as possible,    
where $\ell$ is a given margin-based loss function, $f_\ell^\star=\argmin_{f\in\cF}\E(\ell(Yf(\X)))$ is the optimal decision function and $\cF$ is a set of all real-valued measurable functions on $[0,1]^d$. We assume that $\|f_\ell^\star\|_\infty\le F^\star$ for some $F^\star>0$. This assumption is satisfied if the conditional class probability $\eta(\x)$ satisfies $\inf_{\x\in[0,1]^d}\eta(\x)\wedge(1-\eta(\x))\ge\eta_0$ for some $\eta_0>0$, i.e., $\eta$ is bounded away from 0 and 1, for the exponential and logistic losses. This is because
$f_\ell^\star(\x)=\log( \eta(\x)/(1-\eta(\x))).$
We denote by $\cQ_{F^\star}$ the set of distributions satisfying the above assumption, that is,
    \begin{equation*}
        \cQ_{F^\star}=\cbr{\mbox{Model \labelcref{eq:model_reg_cls}}: \|f_\ell^\star\|_\infty\le F^\star}.
    \end{equation*}

The following theorem states the oracle inequality for the excess risk of the sparse-penalized DNN estimator based on  a strictly convex margin-based loss function.

\begin{theorem}
\label{thm:oracle_cls} 
Let $\ell$ be a strictly convex margin-based loss function with continuous first and second derivatives. Assume that the true generative model $\P$ is in $\cQ_{F^\star}$. Let $F>0$ and let $\{L_n\}_{n\in\bN}$, $\{N_n\}_{n\in\bN}$ and $\{B_n\}_{n\in\bN}$ be positive sequences such that  $L_n\lesssim \log n$, $N_n\lesssim n^{\nu_1}$, $1\le B_n\lesssim n^{\nu_2}$ for some $\nu_1,\nu_2>0$. Then the sparse-penalized DNN estimator defined by
    \begin{equation}
    \label{eq:estimator_cls}
        \hat{f}_n\in\argmin_{f\in\cF^\dnn_n}\sbr{\frac{1}{n}\sum_{i=1}^n\ell(Y_if(\X_i))+\lambda_n\norm{\btheta(f)}_{\clip, \tau_n}}
    \end{equation}
with  $\cF^\dnn_n=\cF^\dnn_\rho(L_n,N_n,B_n,F)$, $\lambda_n\asymp \log^3n/n$ and
$-\log \tau_n \ge A\log^{2}n$ for sufficiently large $A>0$,
satisfies
    \begin{equation}
    \label{eq:oracle_cls}
        \E\sbr{\cE_{\P}(\hat{f}_n)}\le2 
        \inf_{f\in\cF^\dnn_n}\cbr{ \cE_{\P}(f)+\lambda_n\norm{\btheta(f)}_{\clip, \tau_n} } \vee\frac{C\log n}{n},
    \end{equation}
for some universal constant $C>0$, where the expectation is taken over the training data.
\end{theorem}

The following theorem, which is a corollary of \cref{thm:oracle_cls},
is an extension of \cref{cor:conv_reg} for strictly convex margin-based loss functions.

\begin{theorem}
\label{cor:conv_cls}
Let $\{L_n\}_{n\in\bN}$, $\{N_n\}_{n\in\bN}$ and $\{B_n\}_{n\in\bN}$ and $F>0$ be as in \cref{thm:oracle_cls}. Moreover, let $F^\star>0$ and let $\cF^\star$ be a set of some real-valued functions on $[0,1]^d$.  Assume that there are constants $\kappa>0$, $r>0$, $\epsilon_0>0$ and  $C>0$ such that
    \begin{equation}
    \label{eq:approx_cls}
       \sup_{f^\diamond\in\cF^\star:\|f^\diamond\|_\infty\le F^\star}\inf_{f\in\cF^\dnn_\rho(L_n, N_n, B_n, F,  S_{n,\epsilon})}\|f-f^\diamond\|_{2,\P_\X}\le \epsilon,
    \end{equation}
with $S_{n,\epsilon}:=C\epsilon^{-\kappa}\log^rn$ for any $\epsilon\in(0,\epsilon_0)$ and $n\in\bN.$ Then the sparse-penalized DNN estimator defined by  \labelcref{eq:estimator_cls} with $\cF^\dnn_n=\cF^\dnn_\rho(L_n, N_n, B_n, F)$  satisfies
    \begin{equation}
        \sup_{\P\in\cQ_{F^\star}:f_\ell^\star\in\cF^\star} \E\sbr{\cE_{\P}(\hat{f}_n)}
        \lesssim n^{-\frac{2}{\kappa+2}}\log^{3+r}n.
    \end{equation}
\end{theorem}

As is done in \cref{sec:nonpar_reg}, we can obtain the convergence rate of the excess risk using \cref{cor:conv_cls} when the optimal decision function $f_\ell^\star$ belongs to one of the function classes considered in \cref{sec:nonpar_reg}. For example, if the optimal decision function $f_\ell^\star$ is in  H\"older space with smoothness $\alpha>0$, the the excess risk of the sparse-penalized DNN  estimator defined by  \labelcref{eq:estimator_cls} with a piecewise linear activation function and $\nu_2=0$ converges to zero at a  rate $n^{-\frac{2\alpha}{2\alpha+d}}\log^{4}n$.

\section{Computation}
\label{sec:compute}

In this section, we propose a scalable optimization algorithm to solve the problem  \labelcref{eq:estimator}.
Due to the nonlinearity of DNNs, finding the global optimum of \labelcref{eq:estimator}
is almost impossible. There are various gradient based optimization algorithms
which effectively reduce the empirical risk $\cL_n(\btheta):=n^{-1}\sum_{i=1}^n\ell(Y_i, f(\X_i|\btheta))$ \citep{duchi2011adaptive, kingma2014adam, luo2019adaptive, liu2019variance}, where $f(\cdot|\btheta)$ denotes the DNN with parameter $\btheta$. These algorithms, however, would not work well
since not only the empirical risk but also the penalty  are nonconvex. To make the problem simpler, we propose to replace the clipped $L_1$ penalty by its convex tight upper bound. The idea of using the convex upper bound is proposed under the name of the CCCP \citep{yuille2003concave}, the difference of convex functions (DC) programming \citep{tao1997convex} and the majorize-minimization (MM) algorithm \citep{lange2013mm}.

Note that the clipped $L_1$ penalty is decomposed as the sum of the convex and concave parts as
        \begin{equation}
           \|\btheta\|_{\clip,\tau}:=\sum_{j=1}^{p}\left( \frac{|\theta_j|}{\tau}\wedge 1\right)
           =\frac{1}{\tau}\sum_{j=1}^{p} |\theta_j| -\frac{1}{\tau} \sum_{j=1}^{p} (|\theta_j|-\tau)  \mathbbm{1}(|\theta_j| \ge \tau),
\label{eq:clipp-decom}
        \end{equation}
where $p$ denotes the dimension of $\btheta$ and the first term of the right-hand side is convex while the second term is concave in $\btheta$. For given current solution $\hat{\btheta}^{(t)},$ the tight convex upper bound of the second term $-\frac{1}{\tau} \sum_{j=1}^{p} (|\theta_j|-\tau)  \ind(|\theta_j| \ge \tau)$ at the current solution $\hat\theta^{(t)}_j$ is given as
        \begin{equation}
            -\frac{1}{\tau}\sum_{j=1}^{p}\textup{sign}\del{\hat{\theta}^{(t)}_j} \del{\theta_j-\tau}
            \mathbbm{1}\del{|\hat\theta^{(t)}_j|>\tau}.
\label{eq:upper}
        \end{equation}
By replacing $-\frac{1}{\tau} \sum_{j=1}^{p} (|\theta_j|-\tau) \mathbbm{1}(|\theta_j| \ge \tau)$
with the tight convex upper bound \labelcref{eq:upper}, the objective function becomes
        \begin{equation}
        \label{eq:cccp_sub}
        Q^*(\btheta|\hat\btheta^{(t)}):=\cL_n(\btheta)-\inn{\frac{\lambda}{\tau}\h_\tau^{(t)}}{\btheta-\tau\one} + \frac{\lambda}{\tau}\|\btheta\|_1,
        \end{equation}
where
 $$\h_\tau^{(t)}:=\del[1]{\textup{sign}(\hat\btheta^{(t)}_j)\mathbbm{1}\del[0]{|\hat\btheta^{(t)}_j|>\tau}}_{j\in[p]}.$$
The following proposition justifies the use of \labelcref{eq:cccp_sub}.
\begin{proposition}\label{prop:1}
For any parameter $\tilde\btheta$ satisfying $Q^*(\tilde\btheta|\hat\btheta^{(t)})\le Q^*(\hat\btheta^{(t)}|\hat\btheta^{(t)})$, we have $Q(\tilde\btheta)\le Q(\hat\btheta^{(t)}),$
where
    $$Q(\btheta):=\cL_n(\btheta)+ \lambda  \|\btheta\|_{\clip,\tau}.$$
\end{proposition}
\begin{proof}
By definition of $Q^*(\cdot|\hat\btheta^{(t)})$, $Q^*(\hat\btheta^{(t)}|\hat\btheta^{(t)})=Q(\hat\btheta^{(t)})$ and $Q(\tilde\btheta)\le Q^*(\tilde\btheta|\hat\btheta^{(t)})$, which lead to the desired result.
\end{proof}

We apply the proximal gradient descent algorithm \citep{parikh2014proximal} to minimize $Q^*(\btheta|\hat\btheta^{(t)})$. That is, we iteratively update the solution as
    \begin{equations}
     \label{eq:update}
      \hat\btheta^{(t,k+1)}=\argmin_{\btheta}\sbr{\frac{\lambda}{\tau}\|\btheta\|_1+
        \inn{\nabla\cL_n(\hat\btheta^{(t,k)})-\frac{\lambda}{\tau}\h_\tau^{(t)}}{\btheta}
        +\frac{1}{2\eta_t}\norm{\btheta-\hat\btheta^{(t,k)}}_2^2}
    \end{equations}
for $k\in\bN_0$ with $\hat\btheta^{(t,0)}:=\hat\btheta^{(t)},$ where  $\nabla\cL_n(\hat\btheta^{(t,k)})$ is the gradient of $\cL_n(\btheta)$ at $\btheta=\hat\btheta^{(t,k)}$ and $\eta_t$ is a pre-specified learning rate. Then, we let $\hat\btheta^{(t+1)}=\hat\btheta^{(t,k_t^*+1)},$ where
    \begin{equation*}
        k_t^*:=\inf\cbr{k\in\bN_0:Q^*\del[1]{\hat\btheta^{(t,k+1)}|\hat\btheta^{(t)}}\le Q^*\del[1]{\hat\btheta^{(t)}|\hat\btheta^{(t)}}}\wedge \bar{k},
    \end{equation*}
and $\bar{k}$ is the pre-specified maximum number of iterations.
The proximal gradient algorithm is known to reduce
$Q^*(\btheta|\hat\btheta^{(t)})$ well and thus
the proposed algorithm which combines the CCCP and proximal gradient descent algorithm
is expected to decrease the objective function $Q(\btheta)$ monotonically by Proposition \ref{prop:1}.
As an empirical evidence, \cref{fig:obj_ft} draws the curve of the objective function value versus iteration number for a simulated data, which amply shows the monotonicity of our algorithm.

    \begin{figure}
     \centering
        \includegraphics[scale=0.5]{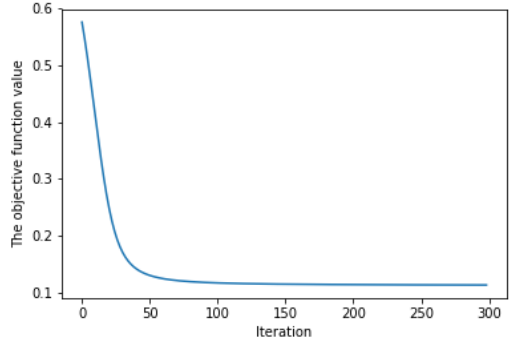}
        \caption{The objective function value versus iteration}
        \label{fig:obj_ft}
     \end{figure}

Note that \labelcref{eq:update} has the closed form solution given as
    \begin{equation}
        \label{eq:cccp_sub_sol}
            \hat\theta^{(t,k+1)}_j=\del{u_{\tau,\lambda,j}^{(t,k)}-\textup{sign}(u_{\tau,\lambda,j}^{(t,k)}) \eta_t\frac{\lambda}{\tau}}\ind\del{\abs{u_{\tau,\lambda,j}^{(t,k)}}\ge \eta_t\frac{\lambda}{\tau}},
        \end{equation}
where
    \begin{equation*}
        u_{\tau,\lambda,j}^{(t,k)}:=\hat\theta_j^{(t,k)} -\eta_t \del{\nabla \cL_n(\hat\btheta^{(t,k)})_j-\frac{\lambda}{\tau}h_{\tau,j}^{(t)}}
    \end{equation*}
for $j\in[p]$. The solution \labelcref{eq:cccp_sub_sol} is a soft-thresholded version of $u_{\tau,\lambda,j}^{(t,k)}$, which is sparse. Thus we can obtain a sparse estimate of the DNN parameter during the training procedure without any post-training pruning algorithm such as \cite{han2015learning, li2016pruning}.

\section{Numerical studies}
\label{sec:numerical}

\subsection{Regression with simulated data}
\label{sec:numerical_reg}

In this section, we carry out simulation studies to illustrate the finite-sample performance of the sparse-penalized DNN estimator (SDNN). We compare the sparse-penalized  DNN estimator with other popularly used regression estimators: kernel ridge regression (KRR), $k$-nearest neighbors (kNN), random forest (RF), and non-sparse DNN (NSDNN). 

For kernel ridge regression we used a radial basis function (RBF) kernel. For both the non-sparse and sparse DNN estimators, we used a network architecture of 5 hidden layers with the numbers of hidden nodes $(100, 100, 100, 100, 100)$. The non-sparse DNN is learned with popularly used optimizing algorithm Adam \citep{kingma2014adam} with learning rate $10^{-3}$. 

We select tuning parameters associated with each estimator by optimizing the performance on a held-out validation data set whose size is one fifth of the size of the training data. 
The tuning parameters include
the scale parameter of the RBF kernel, a degree of regularization for kernel ridge regression, the number of neighbors for  $k$-nearest neighbors,  the depth of the trees for the random forest and
the two tuning parameters $\lambda$ and $\tau$ in the clipped $L_1$ penalty.

We first generate 10-dimensional input $\x$ from the uniform distribution on $[0,1]^{10},$  and
generate  the corresponding response $Y$ from $Y=f^\star(\x)+\epsilon$ for some function $f^\star$, where $\epsilon$ is a standard normal error. The functions used for $f^\star$ are as listed below:
    \begin{align*}
        f_1^\star(\x) &= c_1\sum_{j=1}^{10}(-1)^jj^{-1}x_j\\
        f_2^\star(\x) &= c_2\sin(\|\x\|_1) \\
        f_3^\star(\x) &= c_3\sbr{x_1x_2^2-x_3+\log\del{x_4+4x_5+\exp(x_6x_7-5x_5)}+\tan(x_8+0.1)}\\
        f_4^\star(\x) &= c_4\sbr{\exp\del{3x_1+x_2^2-\sqrt{x_3+5}}+0.01\cot\del{\frac{1}{0.01+|x_4-2x_5+x_6|}}}\\
        f_5^\star(\x) &= c_5\sbr{3\exp (\|\x\|_2)\ind\del{x_2\ge x_3^2 } + x_3^{x_4}-x_5x_6x_7^4}\\
        f_6^\star(\x) &= c_6\big[4x_1x_2x_3x_4\ind\del{x_3+x_4\ge 1, x_5\ge x_6} +\tan(\|\x\|_1)\ind\del{x_1^2x_7x_8\ge x_9x_{10}^3 }\big].
    \end{align*}
The functions $f_1^\star$ and $f_2^\star$ are globally smooth functions, $f_3^\star$ and $f_4^\star$ are composition structured functions and $f_5^\star$ and $f_6^\star$ are piecewise smooth functions. The constants $c_1,\dots, c_6$ are chosen so that the error variance becomes $5\%$ of the variance of the response. 

The performance of each estimator is measured by the empirical $L_2$ error computed based on newly generated $10^5$ simulated data. \cref{fig:boxplot} draws the boxplots of the empirical $L_2$ errors 
of the 5 estimators over 50 simulation replicates for the six true functions.  We see that the sparse-penalized DNN estimator outperforms the other competing estimators for the all 6 true functions, even though it is less stable compared to the other stable estimators (KRR, KNN and RF).

\begin{figure}
    \centering
    \begin{subfigure}[c]{0.45\textwidth}
        \includegraphics[scale=0.45]{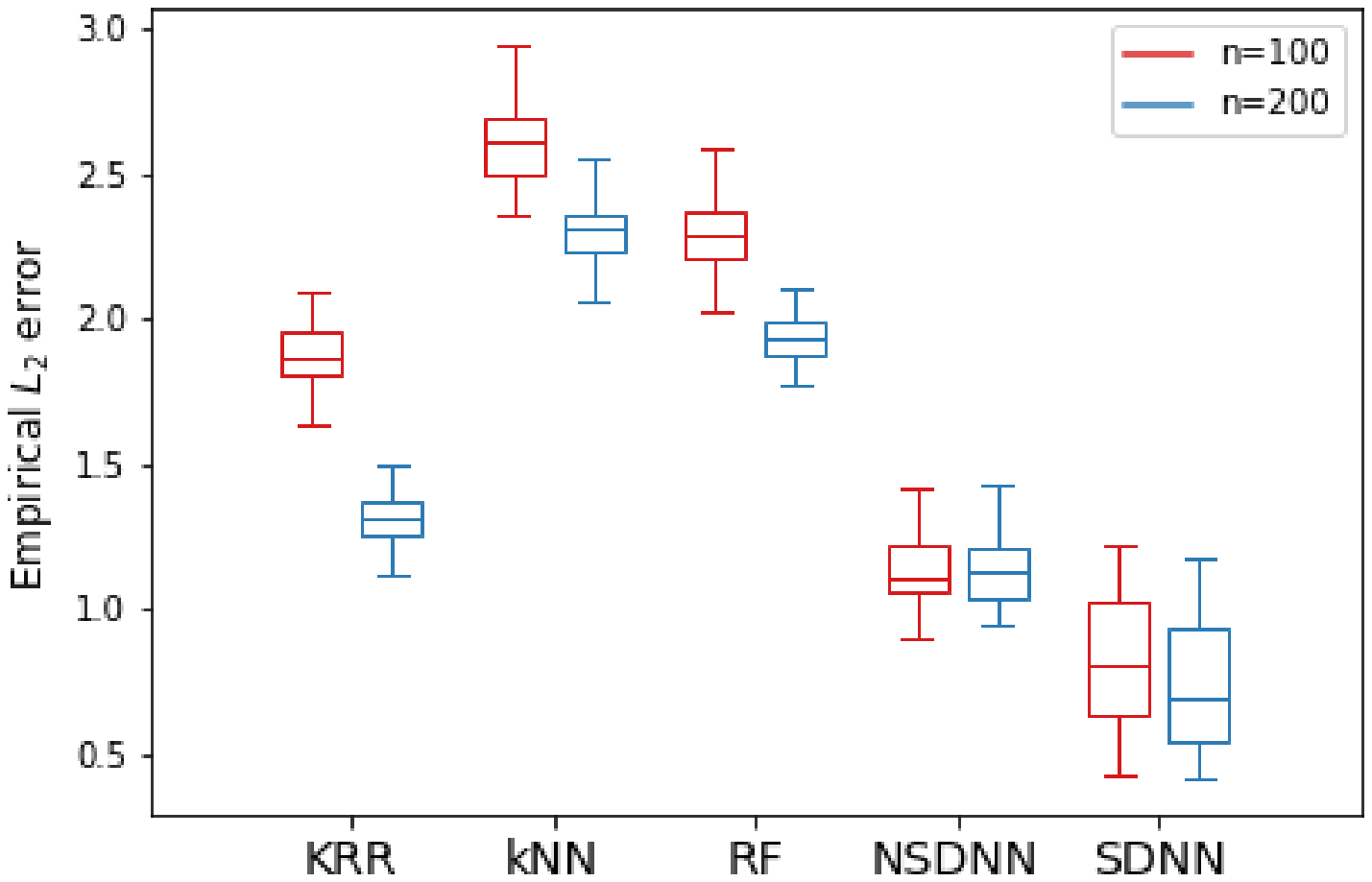}
        \subcaption{$f_1^\star$}
    \end{subfigure}\qquad
    \begin{subfigure}[c]{0.45\textwidth}
        \includegraphics[scale=0.45]{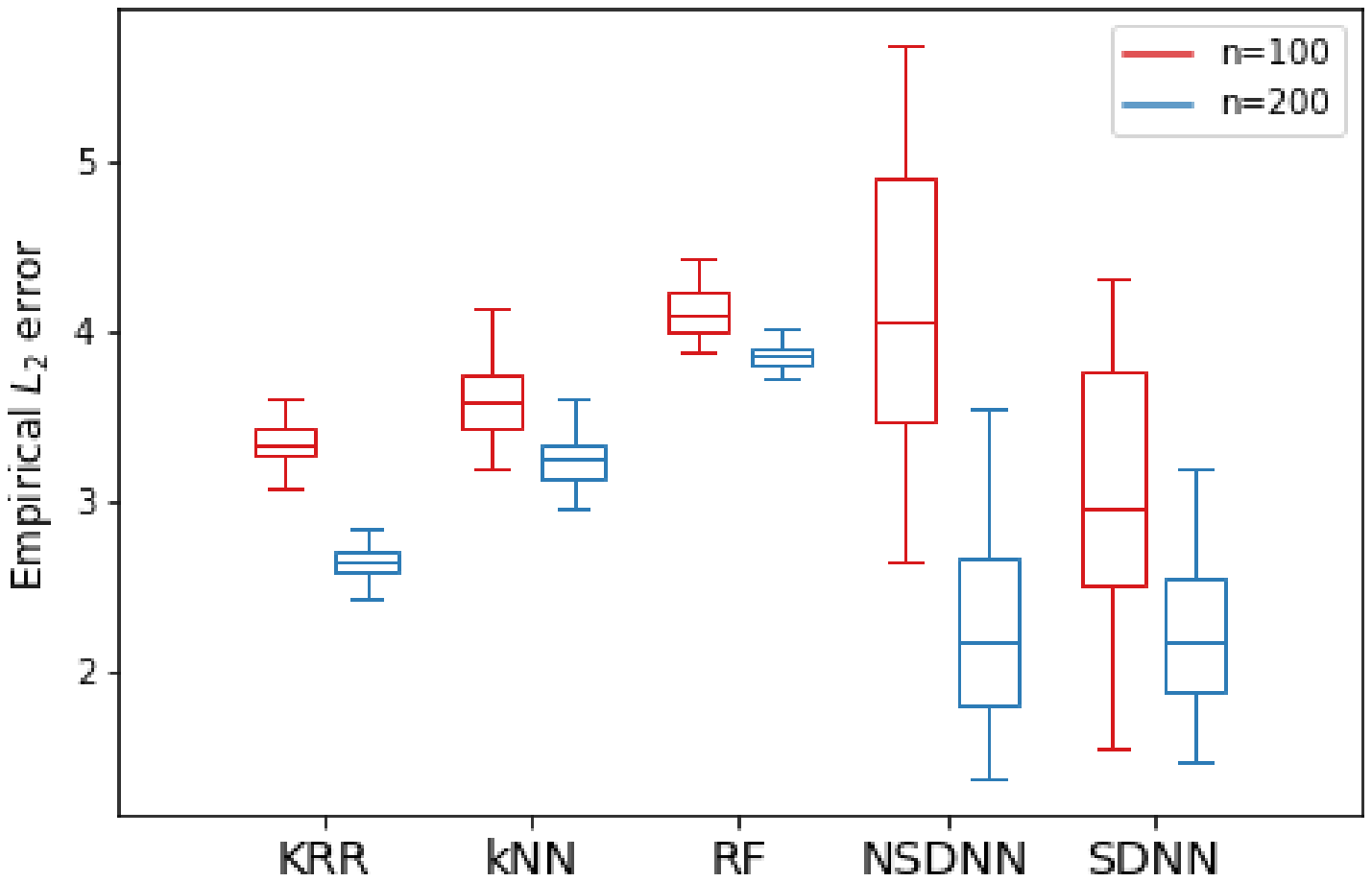}
        \subcaption{$f_2^\star$}
    \end{subfigure}\\
        \begin{subfigure}[c]{0.45\textwidth}
        \includegraphics[scale=0.45]{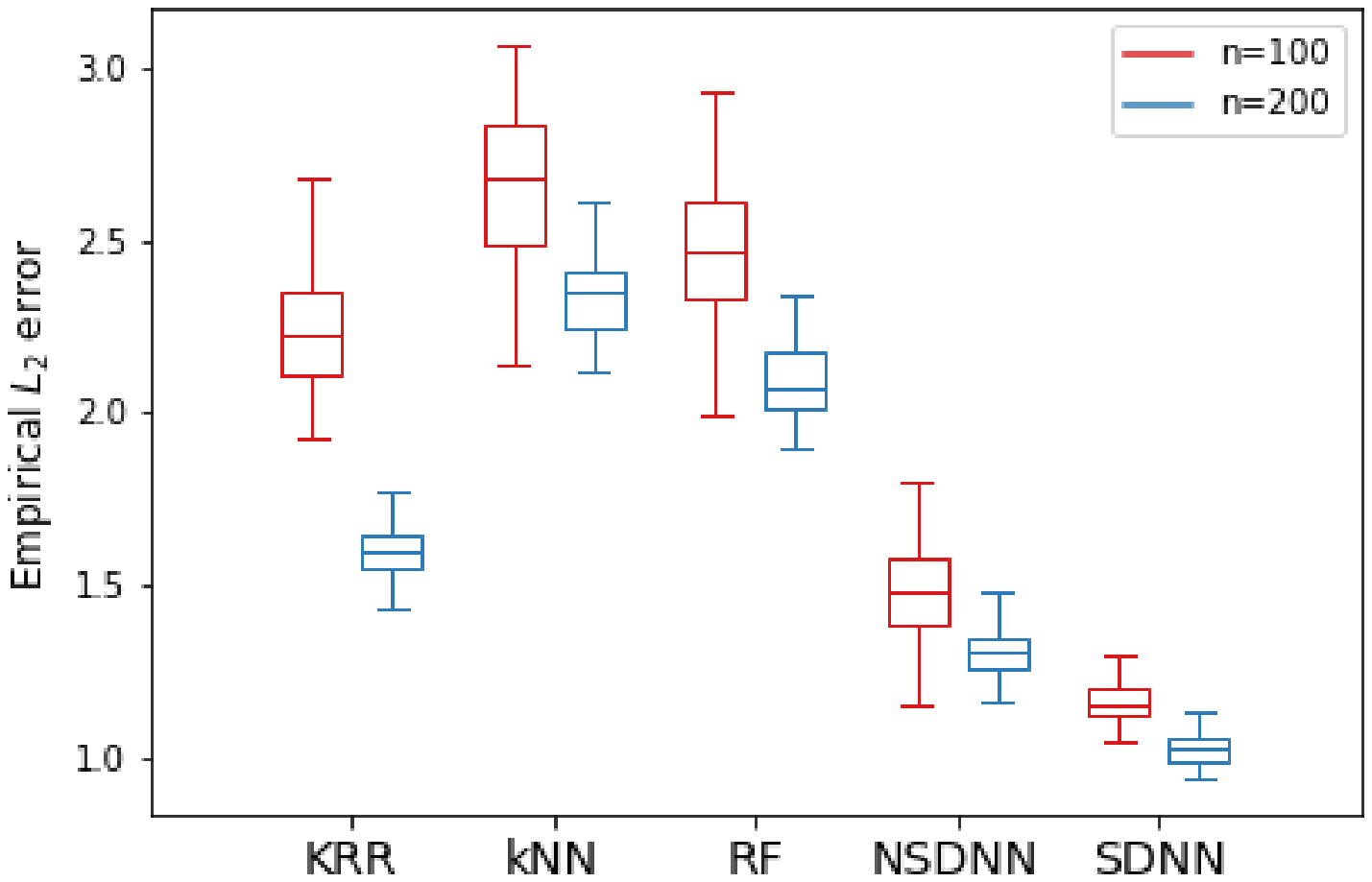}
        \subcaption{$f_3^\star$}
    \end{subfigure}\qquad
    \begin{subfigure}[c]{0.45\textwidth}
        \includegraphics[scale=0.45]{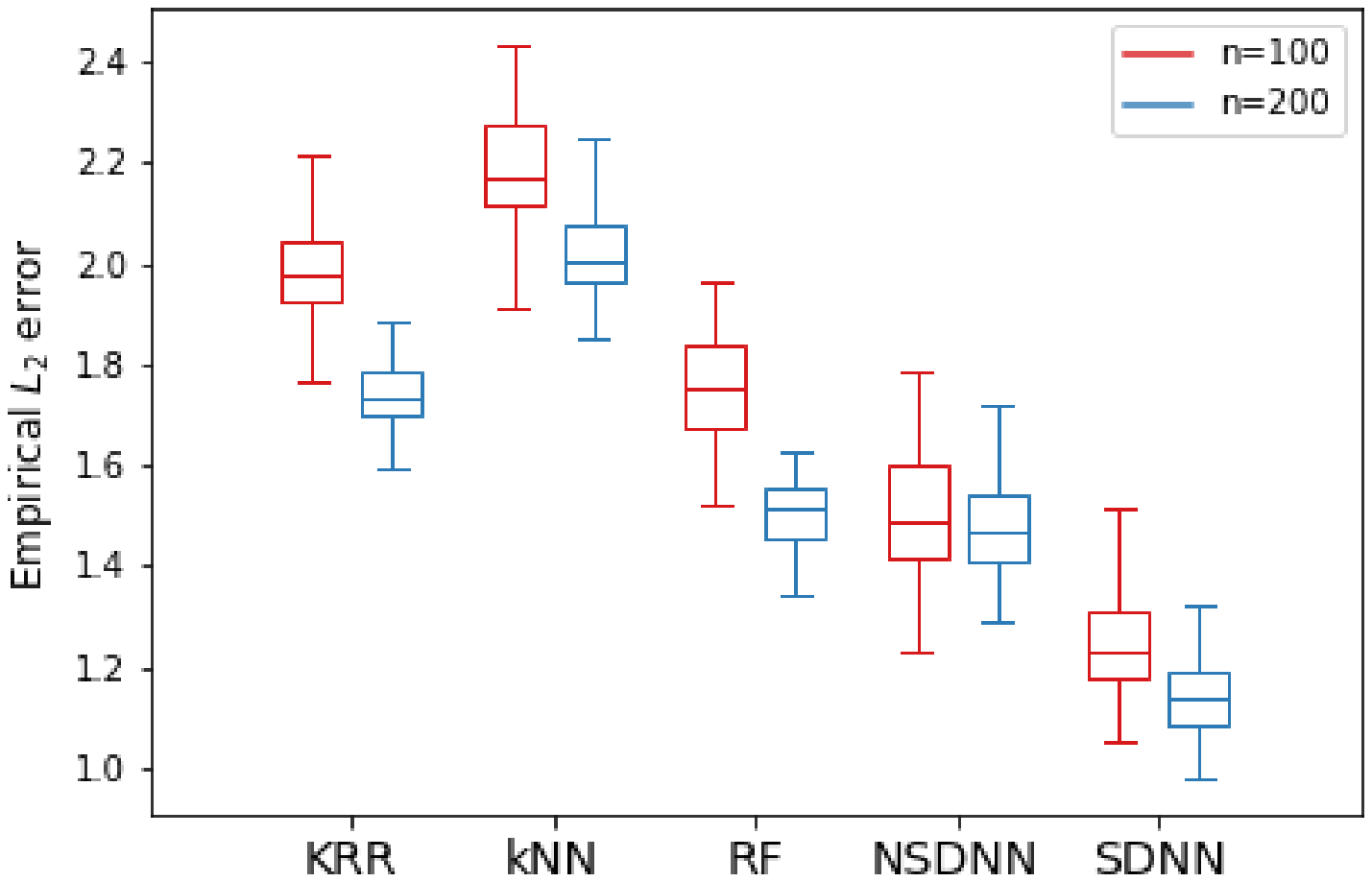}
        \subcaption{$f_4^\star$}
    \end{subfigure}\\
        \begin{subfigure}[c]{0.45\textwidth}
        \includegraphics[scale=0.45]{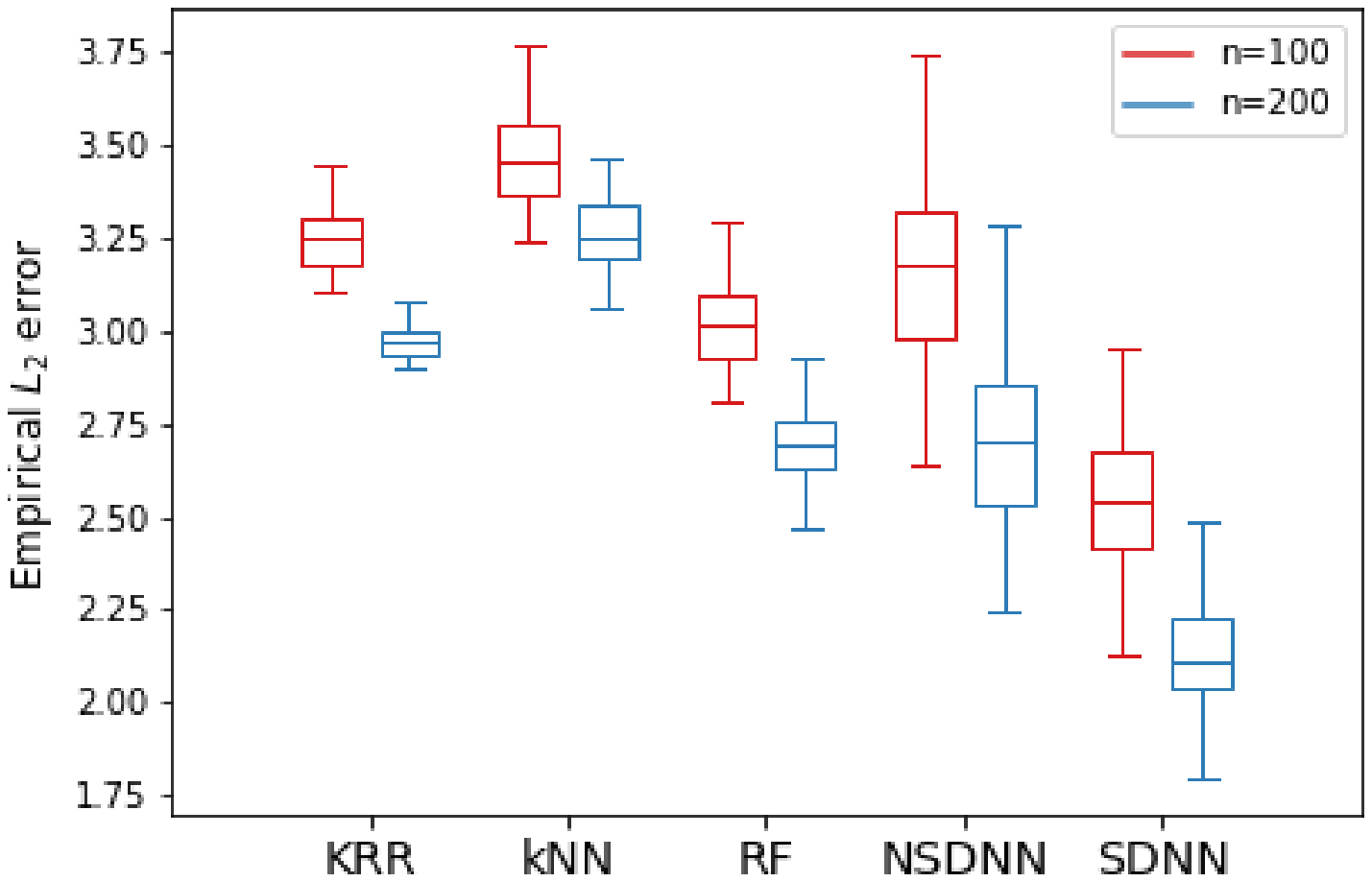}
        \subcaption{$f_5^\star$}
    \end{subfigure}\qquad
    \begin{subfigure}[c]{0.45\textwidth}
        \includegraphics[scale=0.45]{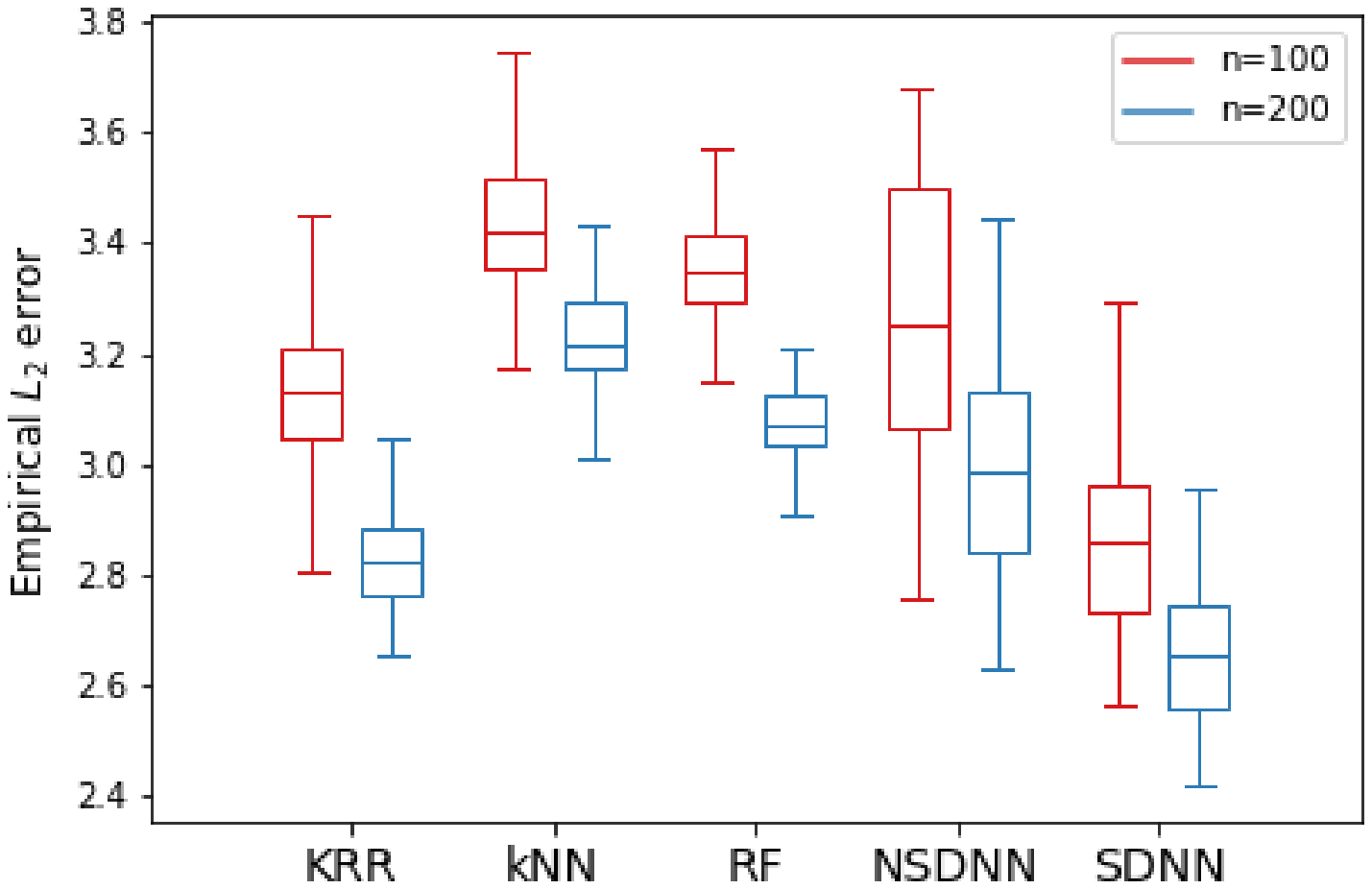}
        \subcaption{$f_6^\star$}
    \end{subfigure}\\
    \caption{Simulation results for the true functions $f_1^\star,\dots, f_6^\star$, respectively. We draw the boxplots of the empirical $L_2$ errors of the 5 estimators from 50 simulation replicates.}
    \label{fig:boxplot}
\end{figure}

\cref{tab:sparsity} presents the sparsity, which is defined as the percentage of non-zero parameters, of the sparse-penalized DNN estimate for each simulation setup. The sparsity ranges from 20\% for estimating the simple linear function $f_1^\star$ to 68\% for estimating the more complex piecewise smooth function $f_5^\star$. This result indicates that  the sparse-penalized DNN estimator can improve the prediction accuracy compared to the non-sparse DNN  by removing redundant parameters adaptively to the ``complexity'' of the true function, as our theory suggests.

\begin{table}[]
\centering
\caption{The averaged sparsity (the percentage of non-zero parameters) and standard deviation in the paranthesis over 50 simulation replicates.}
\label{tab:sparsity}
\begin{tabular}{ccc}
\hline
\multirow{2}{*}{True function} & \multicolumn{2}{c}{Sample size}      \\\cline{2-3}
                               & $n=100$           & $n=200$          \\\hline
$f_1^\star$                    & 22.87\% (4.21\%)  & 20.5\% (3.4\%)   \\
$f_2^\star$                    & 52.96\% (8.67\%)  & 58.27\% (8.69\%) \\
$f_3^\star$                    & 45.23\% (7.88\%)  & 47.16\% (5.24\%) \\
$f_4^\star$                    & 37.38\% (6.95\%)  & 39.29\% (6.93\%) \\
$f_5^\star$                    & 65.94\% (9.12\%)  & 67.43\% (8.65\%) \\
$f_6^\star$                    & 61.56\% (10.64\%) & 67.0\% (8.46\%) \\\hline
\end{tabular}
\end{table}

\subsection{Classification with real data sets}

We compare the sparse-penalized DNN estimator  with other competing estimators by analyzing the following four data sets from the UCI repository:
    \begin{itemize}
        \item Haberman:  Haberman's survival data set contains 306 patients who had undergone surgery for breast cancer at the University of Chicago's Billings Hospital. The task is to predict whether each patient survives after 5 years after the surgery or not.
        \item Retinopathy: This data set contains features extracted from 1,151 eye's images. The task is to predict whether an eye's image contains signs of diabetic retinopathy or not based on the other features.
        \item Tic-tac-toe: This data set contains all the 957 possible board configurations at the end of tic-tac-toe games which are encoded to 27 input variables. The task it to predict the winner of the game.
        \item Promoter: This data set consists of A, C, G, T nucleotides at 57 positions for
106 gene sequences, and each nucleotide is encoded to a 3-dimensional one-hot vector. The task is to predict whether a gene is promoters or non-promoter.
    \end{itemize}

For competing estimators, we considered a support vector machine (SVN), $k$-nearest neighbors (kNN), random forest (RF), and non-sparse DNN (NSDNN). For the support vector machine, we used the RBF kernel. The tuning parameters in each methods are selected by evaluation on a validation data set whose size is one fifth of the size of whole training data. 
    
We splits the whole data into training and test data sets with the ratio 7:3, then evaluate the classification accuracy of each learned estimator on the test data set. We repeat this splits 50 times. \cref{tab:cls_simul} presents the averaged  classification accuracy over 50 training-test splits. The proposed sparse-penalized DNN estimator  is the best for  Tic-tac-toe and Promoter data sets, and the second best for the other two data sets. Moreover, the sparse-penalized DNN estimator is
similarly stable to the other competitors.

\begin{table}[]
\caption{The averaged classification accuracies and standard errors
in the paranthesis over 50 training-test splits of the four UCI data sets.}
\label{tab:cls_simul}
\begin{tabular}{ccccc}
\hline
 Data  & Haberman     & Retinopathy       & Tic-tac-toe      & Promoter        \\\hline
$(n,d)$  & (214, 3)& (805, 19) & (669, 27) & (74, 171) \\ \hline\hline
SVM  & 0.7298 (0.0367) & 0.5737 (0.0282) & 0.8467 (0.0243) & 0.7887 (0.1041) \\
kNN  & \textbf{0.7587 (0.0366)} & 0.6436 (0.0263) & 0.9714 (0.0102) & 0.8012 (0.0649) \\
RF   & 0.7365 (0.0377) & 0.665 (0.0263)  & 0.9777 (0.0103) & 0.8725 (0.0582) \\
NSDNN  & 0.7328 (0.0464) & \textbf{0.7158 (0.0293)} & 0.9735 (0.0107) & 0.8594 (0.062)  \\
SDNN & 0.752 (0.0382)  & 0.6987 (0.0375) & \textbf{0.98 (0.0085) }  & \textbf{0.8769 (0.0474)}
\\\hline
\end{tabular}
\end{table}

To understand the suboptimal accuracy of the sparse-penalized DNN estimator for the two data sets Haberman and Retinopathy, which are
of relatively low input dimensional, we conduct an additional toy experiment that examines an effect of the input dimension. We consider the following probability model for generating simulated data. For a given input dimension $d\in\bN$, let $\X$ be a random vector following the uniform distribution on $[0,1]^d$. Then for a given $\X=\x$, the random variable $Y\in\{-1,1\}$ has the probability mass $\P(Y=1|\X=\x)=(1+\exp(-g^\star(\x)))^{-1}$, where $g^\star$ is a function given by
    \begin{equation*}
        g^\star(\x)=\sum_{j=1}^{\floor{d/2}}x_j^2-\mu
    \end{equation*}
for a constant $\mu\in\R$. We choose $\mu$ so that $\E g^\star(\X)=0$. For each input dimension $d\in\{5,20,35,50\}$, we generate 50 training data sets with size $n=100$ from the above probability model. Then we apply the five estimators considered in this section to the simulated data sets and obtain classification accuracies computed on the test data set independently generated from the same probability model.

The result is presented in \cref{fig:toy}, where the averaged  classification accuracies over 50 simulation replicates for each estimators are reported. For the smallest $d,$  the non-DNN estimators perform  better than both the non-sparse and sparse-penalized DNN estimators. However, as the input dimension increases, the performance of the sparse-penalized DNN estimator is improved quickly and it becomes superior to the other competitors. This result  explains partly why the sparse-penalized DNN estimator does not perform best for the two low dimensional data sets  Haberman $(d=3)$ and Retinopathy $(d=19)$.

\begin{figure}
    \centering
    \includegraphics[scale=0.7]{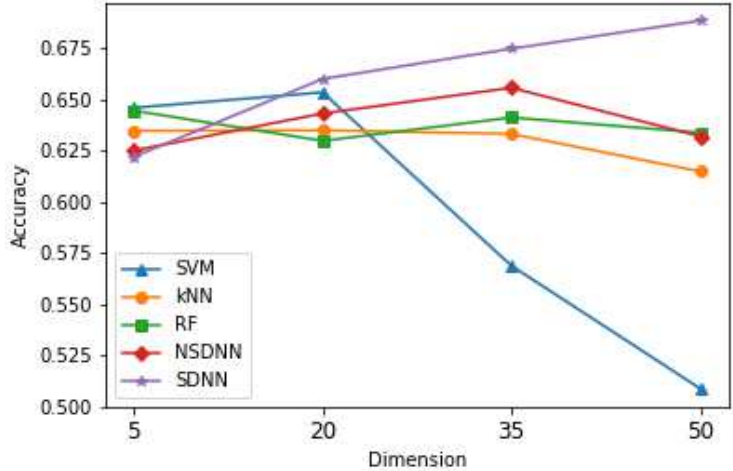}
    \caption{The averaged classification accuracies over 50  simulation replicates with varying input dimensions. }
    \label{fig:toy}
\end{figure}

\section{Conclusion}
\label{sec:conclusion}

In this paper, we proposed a  sparse-penalized DNN estimator leaned with the clipped $L_1$ penalty and
proved the theoretical optimiality.  An interesting conclusion is that the sparse-penalized DNN estimator is extremely flexible so that it achieves the optimal minimax convergence rate (up to a logarithmic factor) without using any information about the true function for various situations.  
Moreover, we proposed an efficient and scalable optimization algorithm so that the sparse-penalized DNN estimator can be used in practice without much difficulty.

There are several possible future works. For binary classification, we only consider the strictly convex losses which are popular in learning DNNs. We have not considered convex but not strictly convex losses such as the hinge loss. We expect that the sparse-penalized DNN estimator learned with the hinge loss and the clipped $L_1$ penalty can attain the minimax optimal convergence rates for estimation of a decision boundary. 

In this paper, we only considered a fully connected DNN. We may use a more structured architecture such as the convolutional neural network when the information of the structure of the true function is available. It would be interesting to investigating how much structured neural networks are helpful compared to simple fully connected neural networks.

Theoretical properties of generative models such as generative adversarial networks \citep{goodfellow2014generative} and variational autoencoders \citep{kingma2013auto} have not been fully studied even though some results are available \citep{liang2018well, briol2019statistical, uppal2019nonparametric}. A difficulty in generative models would be that we have to work with functions where the dimension of the range is larger than the dimension of the domain.

\begin{appendices}
\section{Proofs}
\label{sec:sparse_proofs}

For notational simplicity, we only consider a 1-Lipschitz activation function $\rho$ with $\rho(0)=0$. 
Extensions of the proofs for general $C$-Lipschitz activation functions with arbitrary value of $\rho(0)$ can be done easily.

\subsection{Covering numbers of the DNN classes}

We provide a covering number bound for a class of DNNs with a certain sparsity constraint.
Let $\cF$ be a given class of real-valued functions defined on $\cX$. Let $\delta>0$. A collection $\{f_i:i\in[N]\}$ is called a \textit{$\delta$-covering set} of $\cF$ with respect to the  norm $\|\cdot\|$ if, for all $f\in\cF$, there exists $f_i$ in the collection such that $\|f-f_i\|\le\delta$. The cardinality of the minimal $\delta$-covering set is called the \textit{$\delta$-covering number} of $\cF$ with respect to   the  norm $\|\cdot\|$ , and is denoted by $\cN(\delta, \cF, \|\cdot\|)$.  The following proposition
gives the covering number bound of the class of DNNs with the $L_0$ sparsity constraint.

\begin{proposition}[Proposition 1 of \cite{ohn2019smooth}]
\label{lem:entropy0}
Let $L\in \bN$, $N\in\bN$, $B\ge1$, $F>0$ and $S>0$. Then for any $\delta>0$,
    \begin{equation}
    \log \cN\del{\delta, \cF_{\rho}^\dnn(L, N, B, F, S), \|\cdot\|_\infty} \le 2S(L+1)\log\del{\frac{(L+1)(N+1)B}{\delta}}.
    \end{equation}
\end{proposition}

The following lemma is a technical one.

\begin{lemma}
\label{lem:lip}
Let $L\in \bN$, $N\in\bN$ and $B\ge1.$ For any two DNNs $f_1,f_2\in\cF_\rho(L, N, B, \infty)$, we have
    \begin{equation*}
        \|f_1-f_2\|_{\infty, [0,1]^d} \le (L+1)(B(N+1))^{L+1}\norm{\btheta(f_1)-\btheta(f_2)}_\infty
    \end{equation*}
\end{lemma}

\begin{proof}
For $f\in \cF_\rho(L, N, B, \infty)$ expressed as
    \begin{equation*}
       f(\x)= A_{L+1}\circ\rho_L\circ A_{L}\circ\cdots \circ\rho_1\circ A_1(\x),
    \end{equation*}
we define $[f]_l^-:[0,1]^d\mapsto\R^{N-1}$ and $ [f]_l^+:\R^{N-1}\mapsto\R$ for $l\in \{2,\dots, L\}$ by
    \begin{align*}
                [f]^-_{l}(\cdot)&:=\rho_{l-1}\circ A_{l-1}\circ\cdots \circ\rho_1\circ A_1(\cdot),\\
        [f]^+_{l}(\cdot)&:= A_{L+1}\circ\rho_L\circ A_{L}\circ\cdots  \rho_l\circ\sA_l\circ\rho_{l-1}(\cdot).
    \end{align*}
Corresponding to the last and first layer, we define   $f^-_{1}(\x)=\x$ and $f^+_{L+2}(\x)=\x$. Note that $f=[f]^+_{l+1}\circ A_l \circ [f]^-_{l}$.

Let $\W_l$ and $\b_l$ be the weight matrix and bias vector at the $l$-th hidden layer of $f$. Note that both the numbers of rows and columns of $\W_l$ are less than $N-1$. Thus for any $\x\in[0,1]^d$
    \begin{align*}
        \norm{ [f]^-_{l}(\x)}_\infty 
        &\le \norm{  \W_{l-1}[f]^-_{l-1}(\x) + \b_{l-1}}_\infty \\
        &\le (N+1)B\norm[0]{[f]^-_{l-1}(\x)}_\infty + B\\
        &\le (N+1)B\del{\norm[0]{ [f]^-_{l-1}(\x)}_\infty \vee1}\\
        &\le (N+1)B\del{((N+1)B\norm[0]{ [f]^-_{l-2}(\x)}_\infty \vee1 )\vee1}\\
        &\le ((N+1)B)^2\del{\norm[0]{ [f]^-_{l-2}(\x)}_\infty \vee1}\\
        &\le  ((N+1)B)^{l-1}\del{\norm[0]{\x}_\infty \vee1}\\
        &=((N+1)B)^{l-1},
    \end{align*}
where the fifth inequality follows from  the assumption that $(N+1)B\ge 1$. Similarly, we can show that for any $\z_1,\z_2\in \R^{N}$,
    \begin{align*}
        \abs{ [f]^+_{l+1}(\z_1)- [f]^+_{l+1}(\z_2)}
        \le ((N+1)B)^{L+1-l}\|\z_1-\z_2\|_\infty.
    \end{align*}

For $f_1,f_2\in\cF_\rho(L, N, B, \infty)$, letting $A_{j,l}$ be the affine transform 
at the $l$-th hidden layer of $f_j$ for $j=1,2$, we have for any $\x\in[0,1]^d$,
    \begin{align*}
        |f_1(\x)-f_2(\x)|
        &\le \abs{\sum_{l=1}^{L+1}\sbr{ [f_1]^+_{l+1}\circ A_{1,l} \circ [f_2]^-_{l}(\x) - [f_1]^+_{l+1}\circ A_{2,l} \circ [f_2]^-_{l}(\x)}}\\
        &\le \sum_{l=1}^{L+1} ((N+1)B)^{L+1-l}\norm{(A_{1,l}-A_{2,l}) \circ [f_2]^-_{l}(\x)}_\infty\\
        &\le \sum_{l=1}^{L+1} ((N+1)B)^{L+1-l}\|\btheta(f_1)-\btheta(f_2)\|_\infty\cbr{N\norm{[f_2]^-_{l}(\x)}_\infty+1}\\
         &\le \sum_{l=1}^{L+1} ((N+1)B)^{L+1-l}\|\btheta(f_1)-\btheta(f_2)\|_\infty\cbr{N ((N+1)B)^{l-1}+1}\\
        &\le   (L+1) ((N+1)B)^{L+1}\|\btheta(f_1)-\btheta(f_2)\|_\infty,
    \end{align*}
which completes the proof.
\end{proof}

Using  \cref{lem:entropy0} and \cref{lem:lip}
we can compute an upper bound of the $\delta$-covering number of a class of DNNs with a restriction on the clipped $L_1$ norm when $\delta$ is not too small, which is stated in the following proposition.

\begin{proposition}
\label{lem:entropy_cl1}
Let $L\in \bN$, $N\in\bN$, $B\ge1$, $F>0$ and $\tau>0.$ Let
    \begin{align*}
        \check{\cF}_{\rho, \tau}^\dnn(L, N, B, F, S)&:=\cbr{f\in\cF_\rho^\dnn(L, N, B, F):\|\btheta(f)\|_{\mathsf{clip},\tau}\le S}.
    \end{align*}
 Then we have that for any $\delta>\tau (L+1)((N+1)B)^{L+1}$,
    \begin{equations}
    &\log\cN\del{\delta,  \check{\cF}_{\rho, \tau}^\dnn(L, N, B, F, S), \|\cdot\|_\infty}\\
    &\le \log\cN\del{\delta-\tau (L+1)((N+1)B)^{L+1}, \cF_{\rho}^\dnn(L, N, B, F, S), \|\cdot\|_\infty}\\
     &\le 2S(L+1)\log\del{\frac{(L+1)(N+1)B}{\delta-\tau (L+1)((N+1)B)^{L+1}}}.
    \end{equations}
\end{proposition}

\begin{proof}
For a DNN $f$ with parameter $\btheta(f)$, we let $f^{(\tau)}$ be the DNN constructed by the parameter which is the hard thresholding of $\btheta(f)$ with the threshold $\tau$, that is,  $\btheta(f^{(\tau)})=\btheta(f)\ind(|\btheta(f)|>\tau)$. Then by \cref{lem:lip},
    \begin{align*}
    \norm[0]{f-f^{(\tau)}}_\infty 
        &\le (L+1)((N+1)B)^{L+1}\norm[0]{\btheta(f)-\btheta(f^{(\tau)})}_\infty\\
        &\le \tau (L+1)((N+1)B)^{L+1}. 
    \end{align*}
Given $\delta>(L+1)((N+1)B)^{L+1}$, let $\delta^*:=\delta-(L+1)((N+1)B)^{L+1}>0$ and let $\{f_j^0: j\in [N_{\delta^*}]\}$ be the minimal $\delta^*$-covering set of $\cF_{\rho}^\dnn(L, N, B, F, S)$ with respect to the norm $\|\cdot\|_\infty$, where 
    \begin{equation*}
        N_{\delta^*}:=\cN(\delta^*,\cF_{\rho}^\dnn(L, N, B, F, S), \|\cdot\|_\infty).
    \end{equation*}
Since $\|\btheta(f^{(\tau)})\|_0=\|\btheta(f^{(\tau)})\|_{\mathsf{clip}, \tau}\le \|\btheta(f)\|_{\mathsf{clip}, \tau}\le S$, it follows that $f^{(\tau)}\in\cF_{\rho}^\dnn(L, N, B, F, S)$ for any $f\in \check{\cF}_{\rho, \tau}^\dnn(L, N, B, F, S)$. Hence for any $f\in \check{\cF}_{\rho, \tau}^\dnn(L, N, B, F, S)$, there is $j\in [N_{\delta^*}]$ such that $ \norm[0]{f^{(\tau)}-f_j^0}_\infty\le \delta^*$ and so
    \begin{align*}
        \norm[0]{f-f_j^0}_\infty
        &\le \norm[0]{f-f^{(\tau)}}_\infty +\norm[0]{f^{(\tau)}-f_j^0}_\infty \\
        &\le   \tau (L+1)((N+1)B)^{L+1}+ \delta^*=\delta,
    \end{align*}
which implies that $\{f_j^0: j\in [N_{\delta^*}]\}$ is also a $\delta$-covering set of $ \check{\cF}_{\rho, \tau}^\dnn(L, N, B, F, S)$.
By \cref{lem:entropy0}, the proof is done.
\end{proof}

\subsection{Proofs of \cref{thm:oracle_reg} and  \cref{thm:oracle_cls}}
\label{sec:proof_oracle}

Let $\P_n$ be the empirical distribution based on the data $(\X_1,Y_1),\dots, (\X_n, Y_n)$. We use the abbreviation  $\sQ f:=\int f\d \sQ$ for a measurable function $f$ and measure $\sQ$. Throughout this section, $\cF_n^\dnn:=\cF_\rho^{\dnn}(L_n, N_n, B_n, F)$ and $J_{\lambda_n, \tau_n}(\cdot):=\lambda_n\|\btheta(\cdot)\|_{\clip, \tau_n}$.

For the proofs of \cref{thm:oracle_reg} and  \cref{thm:oracle_cls}, we need the following large deviation bound for empirical processes. This is a slight modification of Theorem 19.3 of \cite{gyorfi2006distribution} that states the result with the covering number with respect to the empirical $L_2$ norm. Since the empirical $L_2$ norm is always less than the $L_\infty$ norm, the following lemma is a direct consequence of Theorem 19.3 of \cite{gyorfi2006distribution}.

\begin{lemma}[Theorem 19.3 of \cite{gyorfi2006distribution}]
\label{lem:emp}
Let $K_1\ge1$ and $K_2\ge1$. Let $\Z_1,\dots, \Z_n$ be independent and identically distributed random variables with values in $\cZ$ and let $\cG$ be a class of functions $g:\cZ\mapsto \R$ with the properties $\|g\|_\infty \le K_1$ and $\E g(\Z)^2\le K_2\E g(\Z)$. Let $\omega\in(0,1)$ and $t^*>0$. Assume that
    \begin{equation}
    \label{eq:lem_cond1}
        \sqrt{n}\omega\sqrt{1-\omega}\sqrt{t^*}\ge 288\max\{2K_1,\sqrt{2K_2}\}
    \end{equation}
and that any $\delta\ge t^*/8$, 
    \begin{equations}
     \label{eq:lem_cond2}
        \frac{\sqrt{n}\omega(1-\omega)\delta}{96\sqrt{2}\max\{K_1,2K_2\}}\ge
        \int_{\frac{\omega(1-\omega)\delta}{16\max\{K_1,2K_2\}}}^{\sqrt{\delta}} \sqrt{\log \cN\del{u, \cG, \|\cdot\|_{\infty}} }\d u.
    \end{equations}
Then
    \begin{equation*}
        \P\del{\sup_{g\in\cG}\frac{\abs{(\P-\P_n)g}}{t^*+\P g}\ge\omega}
        \le 60\exp\del{-\frac{nt^*\omega^2(1-\omega)}{128\cdot2304\max\{K_1^2, K_2\}}}.
    \end{equation*}
\end{lemma}

\begin{proof}[Proof of \cref{thm:oracle_reg}]
Throughout the proof, when comparing two positive sequences $\{a_n\}_{n\in\bN}$ and $\{b_n\}_{n\in\bN}$, we write $a_n\lesssim_{\sigma,F^\star} b_n$ if there is a constant $C_{\sigma,F^\star}>0$ depending only on $\sigma$ and $F^\star$ such that $a_n\le C_{\sigma,F^\star} b_n$ for any $n\in\bN$.

Let $K_n:=(\sqrt{32\sigma^2}\log^{1/2} n)\vee F$. Let $Y^\dag := \text{sign}(Y)(|Y|\wedge K_n),$ which is a truncated version of $Y$ and $f^\dagger$ be the regression function of $Y^\dag$, that is,
    \begin{equation*}
        f^{\dagger}(\x):=\E(Y^\dagger|\X=\x).
    \end{equation*}
We suppress the dependency on $n$ in the notation $Y^\dag$ and $f^\dag$ for notational convenience.
We start with the decomposition
    \begin{align}
       &\norm[0]{\hat{f}_n-f^*}^2_{2, \P_{\X}}
        = \P(Y-\hat{f}_n(\X))^2-\P(Y-f^\star(\X))^2 =\sum_{i=1}^4A_{i,n},
    \end{align}
where
    \begin{align*}
        A_{1,n} &:= \sbr{\P(Y-\hat{f}_n(\X))^2-\P(Y-f^\star(\X))^2}  \\
        &\qquad -\sbr{\P(Y^\dagger-\hat{f}_n(\X))^2-\P(Y^\dagger-f^\dagger(\X))^2} \\
        A_{2,n} &:= \sbr{\P(Y^\dagger-\hat{f}_n(\X))^2-\P(Y^\dagger-f^\dagger(\X))^2} \\ 
        &\qquad -2\sbr{\P_n(Y^\dagger-\hat{f}_n(\X))^2-\P_n(Y^\dagger-f^\dagger(\X))^2} - 2J_{\lambda_n, \tau_n}(\hat{f}_n)\\
        A_{3,n} &:= 2\sbr{\P_n(Y^\dagger-\hat{f}(\X))^2-\P_n(Y^\dagger-f^\dagger(\X))^2}  \\
        &\qquad - 2\sbr{\P_n(Y-\hat{f}(\X))^2-\P_n(Y-f^\star(\X))^2}\\
        A_{4,n} &:= 2\sbr{\P_n(Y-\hat{f}(\X))^2-\P_n(Y-f^\star(\X))^2} + 2J_{\lambda_n,\tau_n}(\hat{f}_n).
    \end{align*}

To bound $A_{1,n}$, we first recall that well-known properties of sub-Gaussian variables such that the condition \labelcref{eq:subgauss} implies that $\P\epsilon=0$ and $\E\e^{\epsilon^2/(4\sigma^2)}\le \sqrt{2}$, e.g., see Theorem 2.6 of \cite{wainwright2019high}. Let
    \begin{align*}
        A_{1,1,n}&:= \P\del{(Y^\dag-Y)(2\hat{f}_n(\X)-Y-Y^\dag)},\\
        A_{1,2,n}&:= \P\cbr{\del{Y^\dag-f^\dag(\X))-Y+f^\star(\X)}
        \del{Y^\dag-f^\dag(\X)+Y-f^\star(\X)}}
    \end{align*}
so that $A_{1,n}=A_{1,1,n}+A_{1,2,n}$. We use the Cauchy-Schwarz inequality to get
    \begin{align*}
        |A_{1,1,n}|
        \le \sqrt{\P\del{Y^\dag-Y}^2}\sqrt{\P\del{2\hat{f}_n(\X)-Y-Y^\dag}^2}.
    \end{align*}
Since $\E\e^{Y^2/(8\sigma^2)}\le \e^{(F^\star)^2/(4\sigma^2)}\E\e^{\epsilon^2/(4\sigma^2)} \le  \sqrt{2}\e^{(F^\star)^2/(4\sigma^2)}$, we have that
     \begin{equations}
     \label{eq:bound_y_ydag}
       \P\del[0]{Y^\dag-Y}^2&=\P[|Y|^2\ind(|Y|>K_n)]\\
       &\le \P[16\sigma^2\e^{Y^2/(16\sigma^2)}\e^{Y^2/(16\sigma^2)-K_n^2/(16\sigma^2)}]\\
       &\le 16\sqrt{2}\sigma^2\e^{(F^\star)^2/(4\sigma^2)}\e^{-2\log n}
       =16\sqrt{2}\sigma^2\e^{(F^\star)^2/(4\sigma^2)}n^{-2},
    \end{equations}
and that
    \begin{equations}
    \label{eq:bound_y_f}
       \P\del{2\hat{f}_n(\X)-Y-Y^\dag}^2
       &\le 2\P(Y^2)+2\P(2\hat{f}_n(\X)-Y^\dag)^2 \\
       &\le 16\sigma^2\E\e^{Y^2/(8\sigma^2)}+18K_n^2 \\
       &\lesssim_{\sigma, F^\star}\log n.
    \end{equations}
Thus $|A_{1,1,n}|\lesssim_{\sigma, F^\star} \log n/n$. For $A_{1,2,n}$, using  the Cauchy-Schwarz inequality we have
    \begin{align*}
        |A_{1,2,n}|
        &\le \sqrt{2\P(Y^\dag-Y)^2+2\P(f^\dag(\X)-f^\star(\X))^2}\\
        &\quad \times\sqrt{\P(Y+Y^\dag-f^\dag(\X)-f^\star(\X))^2}.
    \end{align*}
Using the similar arguments as \labelcref{eq:bound_y_f}, we have $\P(Y+Y^\dag-f^\dag(\X)-f^\star(\X))^2\lesssim_{\sigma, F^\star}\log n$. Since $\P\epsilon=0$, by Jensen's inequality,
    \begin{align*}
        \P(f^\dag(\X)-f^\star(\X))^2
        =\P(\P(Y^\dag|\X)-\P(Y|\X))^2
        \le \P(Y^\dag-Y)^2.
    \end{align*}
Thus the inequality \labelcref{eq:bound_y_ydag} concludes that $|A_{1,2,n}|\lesssim_{\sigma,F^\star} \log n/n$.

The term $\E(A_{3,n})$ can be shown to be bounded above by $\log n/n$ up to a constant depending only on $\sigma$ and $F^\star$ similarly
to the derivations used for bounding $A_{1,n}$.

For $A_{2,n}$, define $\Delta(f)(\Z):=(Y^\dagger-f(\X))^2-(Y^\dagger-f^\dagger(\X))^2$ with $\Z:=(\X, Y)$ for $f\in\cF.$  For $t>0$, we can write
    \begin{align*}
        \P(A_{2,n}>t)
        &\le \P\del{\sup_{f\in\cF^\dnn_n}\frac{(\P-\P_n)\Delta(f)(\Z)}{t+ 2J_{\lambda_n, \tau_n}(f) + \P\Delta(f)(\Z)}\ge \frac{1}{2}}\\
        &\le \sum_{j=0}^\infty\P\del{\sup_{f\in\cF_{n,j,t}}\frac{(\P-\P_n)\Delta(f)(\Z)}{2^jt + \P\Delta(f)(\Z)}\ge \frac{1}{2}},
    \end{align*}
where we define
    \begin{equation*}
        \cF_{n,j, t}:=\cbr{f\in\cF^\dnn_n:2^{j-1}\ind(j\neq0)t\le J_{\lambda_n, \tau_n}(f)\le 2^{j}t}.
    \end{equation*}
We now apply \cref{lem:emp} to the class of functions
        \begin{equation*}
        \cG_{n,j, t}:=\cbr{\Delta(f):[0,1]^d\times\R\mapsto\R :f\in \cF_{n,j, t}}.
    \end{equation*}
We will check the conditions of \cref{lem:emp}. First for sufficiently large $n$, we have that for every $g\in\cG_{n,j,t}$
with $\norm{g}_\infty\le 8K_n^2,$
    \begin{align*}
        \P(g(\Z))^2&=\P(Y^\dagger-f(\X)-(Y^\dagger-f^\dagger(\X)))^2(Y^\dagger-f(\X)+(Y^\dagger-f^\dagger(\X)))^2\\
        &\le4(K_n+F)^2\P(f(\X)-f^\dagger(\X))^2 \\
        &\le 16K_n^2\P g(\Z).
    \end{align*}
Thus, the condition \labelcref{eq:lem_cond1} holds for any sufficiently large $n$ if $t\gtrsim \log^2 n/n$. For the condition \labelcref{eq:lem_cond2}, we observe that 
   \begin{align*}
        &\abs{(y^\dagger-f_1(\x))^2-(y^\dagger-f^\dagger(\x))^2-\cbr{(y^\dagger-f_2(\x))^2-(y^\dagger-f^\dagger(\x))^2}}\\
        &\le |f_1(\x)-f_2(\x)||f_1(\x)+f_2(\x)-2y^\dag| \\
        &\le 4K_n|f_1(\x)-f_2(\x)|
    \end{align*}
for any $f_1, f_2\in\cF_{n,j,t}$ and $(\x,y)\in[0,1]^d\times\R$ and so we have	
    \begin{align*}
       \cN\del{u, \cG_{n,j,t}, \|\cdot\|_\infty}\le \cN\del{u/(4K_n), \cF_{n,j,t}, \|\cdot\|_\infty}.
  \end{align*}
Let $\zeta_n:=(L_n+1)((N_n+1)B_n)^{L_n+1}$. With $\omega=1/2$, by \cref{lem:entropy_cl1} and the assumption that $-\log \tau_n \ge A\log^{2}n$ for sufficiently large $A>0$ which implies $\tau_n\zeta_n\lesssim n^{-1}$, we have that for any $\delta\ge  n^{-1}\log^2n\gtrsim 2^{13}K_n^3\tau_n\zeta_n$,
    \begin{equations}
    \label{eq:ent_bound}
        &\int^{\sqrt{\delta}}_{\delta/(2^{11}K_n^2)} \log^{1/2} \cN\del{\frac{u}{4K_n}, \cF_{n,j,t}, \|\cdot\|_\infty}\d u \\
        &\le \sqrt{\delta} \log^{1/2} \cN\del{\frac{\delta}{2^{13}K_n^3}, \cF_{n,j,t}, \|\cdot\|_\infty}\\
        &\le \sqrt{\delta} \log^{1/2}\cN\del{\frac{\delta}{2^{13}K_n^3}-\tau_n\zeta_n, \cF_{\rho}^{\dnn}(L_n, N_n, B_n, F, 2^jt/\lambda_n), \|\cdot\|_\infty}\\
        &\le 2\sqrt{\delta}(2^jt/\lambda_n)^{1/2} (L_n+1)^{1/2} \log^{1/2} \del{\frac{(L_n+1)(N_n+1)B_n}{\delta/(2^{13}K_n^3)-\tau_n\zeta_n}} \\
        &\le c_{1}\sqrt{\delta}\sqrt{2^jt}\frac{\sqrt{n}}{\log^{3/2}n}
    \end{equations}
for any  sufficiently large $n$ for some constant $c_{1}>0$. Note that the constant $c_{1}$ does not depends on $j$. Then for any $t\ge t_n:=8\log^2n/n$ and any $\delta\ge 2^jt/8\ge\log^2n/n$, there exists an universal constant $c_{2}>0$ such that
\begin{equation}
    \label{eq:delta_cond1}
        \frac{\delta}{\log  n}\ge c_{2}\frac{\sqrt{\delta}\sqrt{2^jt}}{\log^{3/2}n}, 
    \end{equation}
for any $j=0,1,\dots$, and thus the  condition  \labelcref{eq:lem_cond2} is met for any $j=0,1,\dots$ and all sufficiently large $n$.  Therefore we have
    \begin{align*}
    \P(A_{2,n}>t) \lesssim\sum_{j=0}^\infty\exp\del{-c_{3}2^j\frac{nt}{\log n}}
        \lesssim \exp\del{-c_{3}\frac{nt}{\log n}}
    \end{align*}
for $t\ge t_n$, which implies
    \begin{align*}
        \E(A_{2,n})& \le 2t_n + \int_{2t_n}^\infty\P(A_{2,n}>t)\d t\\
        &\lesssim \frac{\log^2n}{n} +\frac{\log n}{n}\e^{-c_{4}\log n} \lesssim \frac{\log^2n}{n}
    \end{align*}
for some positive constants $c_{3}$ and $c_{4}$.

For $A_{4,n}$, we choose a neural network function $f_n^\circ\in\cF^\dnn_n$ such that
    $$\norm{f_n^\circ-f^\star}_{2, \P_\X}^2+ J_{\lambda_n, \tau_n}(f_n^\circ)\le \inf_{f\in\cF^\dnn_n}\sbr{\norm{f-f^\star}_{2,\P_\X}^2 + J_{\lambda_n, \tau_n}(f)} +n^{-1}.$$
Then by the basic inequality $\P_n(Y-\hat{f}_n)^2+J_{\lambda_n,\tau_n}(\hat{f}_n) \le \P_n(Y-f)^2+J_{\lambda_n,\tau_n}(f)$ for any $f\in\cF^\dnn_n$, we have
    \begin{align*}
        A_{4,n} &\le 2\sbr{\P_n(Y-\hat{f}_n(\X))^2-\P_n(Y-f_n^\circ(\X))^2}+ 2J_{\lambda_n,\tau_n}(\hat{f}_n)  \\ &\quad 
        + 2\sbr{\P_n(Y-f_n^\circ(\X))^2-\P_n(Y-f^\star(\X))^2} \\
        &\le 2J_{\lambda_n,\tau_n}(f_n^\circ)+2\sbr{\P_n(Y-f_n^\circ(\X))^2-\P_n(Y-f^\star(\X))^2}
    \end{align*}
and so
    \begin{align*}
       \E(A_{4,n}) &\le 2J_{\lambda_n,\tau_n}(f_n^\circ)+2\norm{f_n^\circ-f^\star}^2_{2,\P_\X}\\
       &\le 2\inf_{f\in\cF^\dnn_n}\sbr{\norm{f-f^\star}^2_{2, \P_\X} + J_{\lambda_n, \tau_n}(f)} +\frac{1}{n}.
    \end{align*}
    
Combining all the bounds we have derived, we get the desired result.
\end{proof}

\begin{proof}[Proof of \cref{thm:oracle_cls}]
Since $\ell$ is continuously differentiable, $\ell$ is Lipschitz on any closed interval. That is, there is a constant $c_1>0$ such that
    \begin{equation}
    \label{eq:losscond1}
        |\ell(z_1)-\ell(z_2)|\le c_1|z_1-z_2|
    \end{equation}
for any $z_1,z_2\in [-F,F]$. On the other hand,  since $F\ge F^\star$, there is a constant $c_2>0$ such that
    \begin{equation}
    \label{eq:losscond2}
        \E\left\{\ell(Yf(\X))-\ell(Yf_\ell^\star(\X))\right\}^2\le c_2\E\left\{\ell(Yf(\X))-\ell(Yf_\ell^\star(\X))\right\}
    \end{equation}
for any  $f\in\{f\in\cF:\|f\|_\infty\le F\}$. This is a well known fact about the strictly convex losses and the proof can be found in Lemma 6.1 of \cite{park2009convergence}.

We decompose $ \cE_{\P}(\hat{f}_n)$ as 
    \begin{align*}
        \cE_{\P}(\hat{f}_n)
        =\P\ell(Y\hat{f}_n(\X)) - \P\ell(Yf_\ell^\star(\X))=B_{1,n}+B_{2,n},
    \end{align*}
where
    \begin{align*}
        B_{1,n}&:=\sbr{\P\ell(Y\hat{f}_n(\X)) - \P\ell(Yf^\star(\X))} \\
        &\quad - 2\sbr{\P_n\ell(Y\hat{f}_n(\X))-\P_n\ell(Yf^\star(\X))}- 2J_{\lambda_n, \tau_n}(\hat{f}_n)\\
        B_{2,n}&:=2\sbr{\P_n\ell(Y\hat{f}_n(\X)) - \P_n\ell(Yf^\star(\X))} + 2J_{\lambda_n, \tau_n}(\hat{f}_n).
    \end{align*}
    
We bound $B_{1,n}$ by using a similar argument for bounding $A_{2,n}$
in the proof of \cref{thm:oracle_reg}. Let  $\Delta(f)(\Z):=\ell(Yf(\X))-\ell(Yf^\star(\X))$ with $\Z:=(\X, Y)$ and let 
    \begin{equation*}
        \cF_{n,j, t}:=\cbr{f\in\cF^\dnn_n:2^{j-1}\ind(j\neq0)t\le J_{\lambda_n, \tau_n}(f)\le 2^{j}t}.
    \end{equation*}
Then for $t>0$, we can write
    \begin{align*}
        \P(B_{1,n}>t)
        &\le \P\del{\sup_{f\in\cF^\dnn_n}\frac{(\P-\P_n)\Delta(f)(\Z)}{t+ 2J_{\lambda_n, \tau_n}(f) + \P\Delta(f)(\Z)}\ge \frac{1}{2}}\\
        &\le \sum_{j=0}^\infty\P\del{\sup_{f\in\cF_{n,j,t}}\frac{(\P-\P_n)\Delta(f)(\Z)}{2^jt + \P\Delta(f)(\Z)}\ge \frac{1}{2}}.
    \end{align*}
We now apply \cref{lem:emp} to the class of functions
        \begin{equation*}
        \cG_{n,j, t}:=\cbr{\Delta(f):[0,1]^d\times\{-1,1\}\mapsto\R :f\in \cF_{n,j, t}},
    \end{equation*}
By \labelcref{eq:losscond1} and \labelcref{eq:losscond2}, we can set $K_1=2c_1F$ and $K_2=c_2$ in  \cref{lem:emp}. The condition \labelcref{eq:lem_cond1} holds for any sufficiently large $n$ if $t\gtrsim \log n/n$. For the condition \labelcref{eq:lem_cond2}, we let $K':=K_1\vee 2K_2= (2c_1F)\vee 2c_2$ for notational simplicity. Further, let $\zeta_n:=(L_n+1)((N_n+1)B_n)^{L_n+1}.$ Then since $\ell$ is locally Lipschitz, using a similar argument to that used for \labelcref{eq:ent_bound} in the proof of \cref{thm:oracle_reg}, we can show that, for any $\delta\ge n^{-1}\log n\gtrsim 4c_2K'\tau_n\zeta_n $,     
 \begin{align*}
        &\int^{\sqrt{\delta}}_{\delta/(4K')} \log^{1/2} \cN\del{u, \cG_{n,j,t}, \|\cdot\|_\infty}\d u \\
        &\le \int^{\sqrt{\delta}}_{\delta/(4K')} \log^{1/2}\cN\del{u/c_2, \cG_{n,j,t}, \|\cdot\|_\infty}\d u \\
        &\le \sqrt{\delta}(2^jt/\lambda_n)^{1/2} L_n^{1/2} \log^{1/2} \del{\frac{(L_n+1)(N_n+1)B_n}{\delta/(4c_2K')-\tau_n\zeta_n}} \\
        &\le c_{3}\sqrt{\delta}\sqrt{2^jt}\frac{\sqrt{n}}{\log^{1/2}n}
    \end{align*}
for all sufficiently large $n$ for some constant $c_{3}>0$. For any $t\ge t_n:=8n^{-1}\log n$ and any $\delta\ge 2^jt/8\ge n^{-1}\log n$, there exists a constant $c_{4}>0$ such that $\delta\ge c_{4} \sqrt{\delta}\sqrt{2^jt}/\log^{1/2}n$ for any $j=0,1,\dots,$. Thus condition  \labelcref{eq:lem_cond2} is met and then we have
        \begin{align*}
        \E(B_{1,n})& \le 2t_n + \int_{2t_n}^\infty\P(B_{1,n}>t)\d t\\
        &\le 2t_n + \int_{2t_n}^\infty\exp(-c_{5}nt)\d t\\
        &\lesssim \frac{\log n}{n} +\frac{1}{n}\e^{-c_{6}\log n} \\
        &\lesssim \frac{\log n}{n}
    \end{align*}
for some positive constants $c_{5}$ and $c_{6}$.

For $B_{2,n}$, we choose a neural network function $f_n^\circ\in\cF^\dnn_n$ such that
    $$\cE_\P(f_n^\circ)+ J_{\lambda_n, \tau_n}(f_n^\circ)\le \inf_{f\in\cF^\dnn_n}\sbr{\cE_\P(f) + J_{\lambda_n, \tau_n}(f)} +n^{-1}.$$
Then by the basic inequality $\P_n\ell(Y\hat{f}_n(\X))+J_{\lambda_n,\tau_n}(\hat{f}_n) \le \P_n\ell(Yf(\X))+J_{\lambda_n,\tau_n}(f)$ for any $f\in\cF^\dnn_n$, we have
    \begin{align*}
        B_{2,n} &\le 2\sbr{\P_n\ell(Y\hat{f}_n(\X))-\P_n\ell(Yf_n^\circ(\X))}+ 2J_{\lambda_n,\tau_n}(\hat{f}_n)  \\ &\quad 
        + 2\sbr{\P_n\ell(Y f_n^\circ(\X))-\P_n\ell(Yf^\star(\X))} \\
        &\le 2J_{\lambda_n,\tau_n}(f_n^\circ)+2\sbr{\P_n\ell(Y f_n^\circ(\X))-\P_n\ell(Y, f^\star(\X))}
    \end{align*}
and so
    \begin{align*}
       \E(B_{2,n}) &\le 2J_{\lambda_n,\tau_n}(f_n^\circ)+2\cE_\P(f_n^\circ)\le 2\inf_{f\in\cF^\dnn_n}\sbr{\cE_\P(f) + J_{\lambda_n, \tau_n}(f)} +\frac{1}{n}.
    \end{align*}
    
Combining all the bounds we have derived, we get the desired result.
\end{proof}

\subsection{Proofs of \cref{cor:conv_reg} and  \cref{cor:conv_cls}}
\label{sec:proof_cor}

\begin{proof}[Proof of \cref{cor:conv_reg}]
Let $\epsilon_n=n^{-\frac{1}{\kappa+2}}$.  By \cref{thm:oracle_reg}, the assumption \labelcref{eq:approx_reg}, and the fact that $ \|\btheta(f)\|_{\clip, \tau}\le \|\btheta(f)\|_0$ for any $\tau>0$, we have that for any $f^\star\in\cF^\star$,
    \begin{align*}
         &\E\sbr{\norm[0]{\hat{f}_n-f^\star}_{2,\P_\X}^2}\\
         &\lesssim \inf_{f\in\cF^\dnn_\rho(L_n, N_n, B_n, F,  C\epsilon_n^{-\kappa}\log^rn)}\cbr{ \norm{f-f^\star}_{2,\P_\X}^2+\lambda_n\|\btheta(f)\|_{\clip, \tau_n}} \vee \frac{\log^2n}{n} \\
         &\lesssim \sbr{\inf_{f\in\cF^\dnn_\rho(L_n, N_n, B_n, F,  C\epsilon_n^{-\kappa}\log^rn)} \norm{f-f^\star}_{2,\P_\X}^2 + \lambda_n\epsilon_n^{-\kappa}\log^rn} \vee \frac{\log^2n}{n} \\
         &\lesssim \sbr{\epsilon_n^2+n^{\frac{\kappa}{\kappa+2}}\frac{\log^{5+r}n}{n} }\vee \frac{\log^2n}{n} \\
         &\lesssim n^{-\frac{2}{\kappa+2}}\log^{5+r}n
    \end{align*}
which concludes the desired result.
\end{proof}

\begin{proof}[Proof of \cref{cor:conv_cls}]
For $\x\in[0,1]^d$, define the function $\psi_\x:\R\mapsto\R_+$ by $\psi_\x(z):=\eta(\x)\ell(z)-(1-\eta(\x))\ell(-z).$ Note that $z_{\x}^\star:=f_\ell^\star(\x)$ is the minimizer of $\psi_\x(z)$ and satisfies $\psi_\x'(z_\x^\star)=0$ $\P_\X$-a.s.. Then by the Taylor expansion around $z_{\x}^\star:=f_\ell^\star(\x),$
we have
    \begin{equation*}
        \psi_\x(z)-\psi_\x(z_{\x}^\star)=\frac{\psi''(\tilde{z})}{2}(z-z_{\x}^\star)^2
    \end{equation*}
$\P_\X$-a.s., where $\tilde{z}$ lies between $z$ and $z_\x^\star$. Since $\ell$ has a continuous second derivative, we have $\|\psi''\|_{\infty, [-F,F]}\le c_{1}$ for some $c_{1}>0$, which implies that
    \begin{equation*}
        \cE_\P(f)\le c_{1}\|f-f_\ell^\star\|_{2, \P_\X}^2.
    \end{equation*}

Let $\epsilon_n=n^{-\frac{1}{\kappa+2}}$.  By \cref{thm:oracle_cls}, the condition \labelcref{eq:approx_cls}, and the fact that $ \|\btheta(f)\|_{\clip, \tau}\le \|\btheta(f)\|_0$ for any $\tau>0$, we have that for any $f_\ell^\star\in\cF^\star$,
    \begin{align*}
         \E\sbr{\cE_\P(\hat{f}_n)}
          &\lesssim \inf_{f\in\cF^\dnn_\rho(L_n, N_n, B_n, F,  C\epsilon_n^{-\kappa}\log^rn)}\cbr{ \cE_\P(f)+\lambda_n\|\btheta(f)\|_{\clip, \tau_n}} \vee \frac{\log^2n}{n} \\
         &\lesssim \sbr{\inf_{f\in\cF^\dnn_\rho(L_n, N_n, B_n, F,  C\epsilon_n^{-\kappa}\log^rn)} \norm{f-f_\ell^\star}_{2,\P_\X}^2 + \lambda_n\epsilon_n^{-\kappa}\log^rn }\vee \frac{\log^2 n}{n} \\
         &\lesssim n^{-\frac{2}{\kappa+2}}\log^{3+r}n
    \end{align*}
which concludes the desired result.
\end{proof}

\section{Function approximation by a DNN with general activation functions}
\label{sec:activation}

\subsection{Examples of activation functions}
\label{sec:activation_examples}

Examples of piecewise linear activation functions are
    \begin{itemize}
        \item ReLU $z\mapsto \max\{z,0\}$ 
        \item Leaky ReLU $z\mapsto \max\{z,az\}$ for $a\in(0,1),$
    \end{itemize}
and examples of locally quadratic activation functions are 
    \begin{itemize}
        \item Sigmoid: $\displaystyle z\mapsto 1/({1+\e^{-z}}).$ 
        \item Tangent hyperbolic: $\displaystyle  z\mapsto\frac{\e^z-\e^{-z}}{\e^z+\e^{-z}}.$
        \item Inverse square root unit (ISRU) \citep{carlile2017improving}: 
            $\displaystyle  z\mapsto\frac{z}{\sqrt{1+az^2}}$ for $a>0$. 
        \item Soft clipping \citep{klimek2018neural}: 
            $\displaystyle  z\mapsto\frac{1}{a}\log \del{\frac{1+\e^{az}}{1+\e^{a(z-1)}}}$ for $a>0$.  \medskip
        \item SoftPlus \citep{glorot2011deep}: $ z\mapsto\log (1+\e^z)$.  
        \item Swish \citep{ramachandran2017searching}: $\displaystyle  z\mapsto\frac{z}{1+\e^{-z}}$.
        \item Exponential linear unit (ELU) \citep{clevert2015fast}: 
            $\displaystyle  z\mapsto a(\e^z-1)\ind(z\le0)+ z\ind(z>0)$ for $a>0$.
		\item Inverse square root linear unit  (ISRLU) \citep{carlile2017improving}: 
            $\displaystyle   z\mapsto\frac{z}{\sqrt{1+az^2}}\ind({z\le0})+z\ind({z>0})$ for $a>0$.
        \item Softsign \citep{bergstra2009quadratic}: 
            $\displaystyle  z\mapsto \frac{z}{1+|z|}.$  
    \end{itemize}

\subsection{Approximation of H\"older smooth functions}

The next theorem present the result about the approximation of H\"older smooth function by a DNN with the activation function being either piecewise linear or locally quadratic.

\begin{theorem}
\label{thm:holder_approx}
Let $f^\star\in\cH^{\alpha, R}([0,1]^d)$. Then there exist positive constants $L_0$, $N_0$, $S_0$, $B_0$ and $F_0$ depending only on $d$, $\alpha$, $R$ and $\rho(\cdot)$ such that, for any $\epsilon>0$, there is a neural network 
    \begin{equation*}
        f\in  \cF^\dnn_\rho\del{L_0\log(1/\epsilon),N_0\epsilon^{-d/\alpha}, B_0, F_0, S_0\epsilon^{-d/\alpha}\log(1/\epsilon)}
    \end{equation*}
    for a piecewsie linear $\rho$ and
\begin{equation*}
        f\in  \cF^\dnn_\rho\del{L_0\log(1/\epsilon),N_0\epsilon^{-d/\alpha}, B_0\epsilon^{-4(d/\alpha+1)}, F_0, S_0\epsilon^{-d/\alpha}\log(1/\epsilon)}
    \end{equation*}
      for a locally quadratic $\rho$
satisfying
    \begin{equation*}
        \|f^\star-f\|_\infty\le \epsilon.
    \end{equation*}
\end{theorem}

\begin{proof}
See Theorem 1 of \cite{ohn2019smooth}.
\end{proof}

\subsection{Approximation of composition structured functions}

Recall the class of composition structure functions $\cG^{\textsc{comp}}(q, \balpha,  \mathbf{d}, \t, R)$  given in \labelcref{eq:compose}. For the proof of the  approximation result of a composition structured function by a DNN with the activation function being either piecewise linear of locally quadratic, we need following lemma.

\begin{lemma}[Lemma 3 of \cite{schmidt2020nonparametric}]
\label{eq:composite_ineq}
Let $f^\star:=g_q^\star\circ\dots\circ g_1^\star\in\cG^{\textsc{comp}}(q, \balpha,  \mathbf{d}, \t, R)$. Then for any $g_j\equiv(g_{jk})_{k\in[d_{j+1}]}$ with $g_{jk}$ being a real-valued function for $j\in[q]$, we have
    \begin{equation*}
        \norm{g_q^\star\circ \cdots\circ g_1^\star-g_q\circ \cdots\circ g_1}_{\infty}\le R(2R)^{\sum_{j=1}^{q-1}\alpha_{j+1}}\sum_{j=1}^q\max_{1\le k\le d_{j+1}}\norm{g_{jk}^\star-g_{jk}}_\infty^{\prod_{h=j+1}^q \alpha_h\wedge 1}.
    \end{equation*}
\end{lemma}

The following lemma is used to prove the approximation result with locally quadratic activation function.

\begin{lemma}
\label{lem:identity_approx}
Let the activation function $\rho(\cdot)$ be locally quadratic. Let $\delta\ge0$. Then for any $\epsilon>0$, there exists a DNN $f_{\textup{id},\delta}\in\cF_\rho^\dnn(1, 3, C_1(1+\delta)^2\epsilon^{-1}, 1+\delta+\epsilon)$ such that 
    \begin{equation*}
        \sup_{\x\in[-\delta,1+\delta]}|f_{\textup{id}, \delta}(\x)-\x|\le \epsilon
    \end{equation*}
for some constants $C_1>0$ depending only on $\rho(\cdot)$
\end{lemma}

\begin{proof}
Consider a DNN such that
 $$f_{\textup{id}, \delta}(x):=\frac{K}{\rho'(t)}\sbr{\rho\del{\frac{1}{K}x + t}-\rho(t)}$$
for some $K>0$ that will be defined later. Then by  Taylor expansion around $t$, we have
    \begin{align*}
        f_{\textup{id},\delta}(x)
        =\frac{K}{\rho'(t)}\sbr{\frac{\rho'(t)}{K}x + \frac{\rho''(\tilde{x})}{2K^2}x^2}= x+  \frac{\rho''(\tilde{x})}{\rho'(t)K}x^2
    \end{align*}
where $\tilde{x}$ lies between $t$ and $x$. Since the second derivative of $\rho$ is bounded and $\rho'(t)>0,$ we have
    \begin{equation*}
        \sup_{\x\in[-\delta,1+\delta]}|f_{\textup{id},\delta}(\x)-\x|\le C_2(1+\delta)^2/K.
    \end{equation*}
for some constant $C_2>0$ depending only on the activation function $\rho(\cdot)$. Taking $K=C_1(1+\delta)^2/\epsilon $, we have the desired result.
\end{proof}

\begin{theorem}
\label{thm:composite_approx}
Let $\epsilon_0>0$. Let $f^\star:=g_q^\star\circ\dots\circ g_1^\star\in\cG^{\textsc{comp}}(q,\balpha, \mathbf{d}, \t, R)$. Let  $\alpha_j^*:=\alpha_j\prod_{h=j+1}^q\alpha _h\wedge 1$ for $j\in[q]$ and $\alpha_{\min}:=\min_{j\in[q]}\alpha_j>0$. Let $\kappa:=\max_{j\in[q]}t_j/\alpha_j^*$. Then there exist positive constants $L_0$, $N_0$, $S_0$, $B_0$ and $F_0$ depending only on $\balpha$, $\mathbf{d}$, $\mathbf{t}$, $R$, $\epsilon_0$ and $\rho(\cdot)$ such that, for any $\epsilon\in(0,\epsilon_0)$, there is a DNN
    \begin{equation}
        f\in  \cF^\dnn_\rho\del{L_0\log(1/\epsilon),N_0\epsilon^{-\kappa}, B_0, F_0, S_0\epsilon^{-\kappa}\log(1/\epsilon)}
    \end{equation}
for a piecewise linear $\rho$ and
    \begin{equation}
            f\in  \cF^\dnn_\rho\big(L_0\log(1/\epsilon),N_0\epsilon^{-\kappa}, B_0\epsilon^{-4\kappa-4-\alpha_{\min}^{-q+1}},F_0, S_0\epsilon^{-\kappa}\log(1/\epsilon)\big)
    \end{equation}
for a locally quadratic $\rho$    satisfying
    \begin{equation}
        \norm{f^\star-f}_\infty\le \epsilon.
    \end{equation}
\end{theorem}

\begin{proof}
For a piecwise linear activation function, combining Lemma A1 of \cite{ohn2019smooth}, Theorem 5 and Lemma 3 of \cite{schmidt2020nonparametric}, we obtain the desired result.

For a locally quadratic activation function, without loss of generality,  we assume that $a_j=0$ and $b_j=1$ for all $j\in [q]$. By \cref{thm:holder_approx}, for each $j\in[q]$, $k\in[d_{j+1}]$, and for any $\epsilon>0$, there is a DNN
    \begin{equation*}
        \tilde{g}_{jk}\in\cF_\rho^\dnn\del{L_{0,j}\log(1/\epsilon), N_{0,j}\epsilon^{-t_j/\alpha_j^*}, B_{0,j}\epsilon^{-4(t_j/\alpha_j^*+1)}, F_{0,j}, S_{0,j}\epsilon^{-t_j/\alpha_j^*}\log(1/\epsilon)}
    \end{equation*}
such that 
    \begin{equation*}
        \|g^\star_{jk}-\tilde{g}_{jk}\|_\infty\le \epsilon^{\alpha_j/\alpha_j^*}
    \end{equation*}
for some positive constants $L_{0,j}, N_{0,j}, B_{0,j}, F_{0,j}$ and  $S_{0,j}$. On the other hand, for every $j\in[q]$, \cref{lem:identity_approx} implies that there is a DNN $f_{\text{id},j}\in\cF_\rho^\dnn(1, 3, C_1\epsilon^{-\alpha_j/\alpha_j^*}, 2)$ such that
    \begin{equation*}
        \|f_{\text{id},j}\circ \tilde{g}_{jk}-\tilde{g}_{jk}\|_\infty\le \epsilon^{\alpha_j/\alpha_j^*}
    \end{equation*}
for some constant $C_1>0$ depending only on $\rho(\cdot)$, $\epsilon_0$ and $\alpha_{\min}$, and thus $ \|g^\star_{jk}-f_{\text{id},j}\circ \tilde{g}_{jk}\|_\infty\le 2\epsilon^{\alpha_j/\alpha_j^*}$.  For approximation of $g^\star_{jk}$, we consider the DNN $ g_{jk}:=f_{\text{id},j}\circ \tilde{g}_{jk}$ instead of $\tilde{g}_{jk}$ in order to control the sparsity of the composited DNNs. Note that 
    \begin{align*}
        \|\btheta(g_j\circ g_{j-1})\|_0
        &\le \sum_{k=1}^{d_{j+1}}\|f_{\text{id},j}\|_0\|\btheta(\tilde{g}_{jk})\|_0+ \sum_{k=1}^{d_{j}}\|f_{\text{id},j}\|_0\|\btheta(\tilde{g}_{(j-1)k})\|_0 \\
        &\lesssim (\epsilon^{-t_j/\alpha_j^*}\vee \epsilon^{-t_{j-1}/\alpha_{j-1}^*})\log(1/\epsilon),
    \end{align*}
where $g_j:=(g_{jk})_{k\in [d_{j+1}]}$ for $j\in [q]$. Let $f:=g_q\circ\dots \circ g_1$. Then we have $   \|\btheta(f)\|_0\lesssim \max_{j\in[q]}\epsilon^{-t_j/\alpha_j^*}\log(1/\epsilon)$,
    \begin{align*}
        \|\btheta(f)\|_\infty
        &\le \max_{j\in[q]}\cbr{\max_{k\in[d_{j+1}]}\|\btheta(g_{jk})\|_0\vee \|\btheta(f_{\text{id}, j})\|_0}\\
        &\lesssim \max_{j\in [q]}\epsilon^{-4(t_j/\alpha_j^*+1)}\epsilon^{-\alpha_j/\alpha_j^\star}\\
        &\le \max_{j\in [q]}\epsilon^{-4(t_j/\alpha_j^*+1)-\alpha_{\min}^{-q+1}}
    \end{align*}
and
    \begin{align*}
        \|f^\star-f\|_\infty 
        \le C_2\sum_{j=1}^q\max_{1\le k\le d_{j+1}}\|g^\star_{jk}-g_{jk}\|_\infty^{\alpha_j^*/\alpha_j}
        \le qC_2\epsilon
        \end{align*}
for some constant $C_2>0$ by \cref{eq:composite_ineq}, which completes the proof.
\end{proof}

\subsection{Approximation of piecewise smooth functions}

In this section, we consider the approximation of piecewise smooth functions $\cG^{\textsc{piece}}(\alpha, \beta, M, K, R)$ given in \labelcref{eq:piece} by a DNN with the activation function being either peicewise linear and locally quadratic. We need following lemma for the proof for locally quadratic activation functions.

\begin{lemma}
\label{lem:indicator_approx}
Let the activation function $\rho(\cdot)$ be locally quadratic. There is a DNN $f_{\textup{ind}}\in \cF_\rho (\ceil{C_1\log (1/\epsilon)}, 31, \epsilon^{-8}\vee C_2, 2)$ such that
    \begin{align*}
       \int_{-1}^1|f_{\textup{ind}}(x)-\ind(x\ge 0)|^2\d x\le \epsilon
    \end{align*}
for some constants $C_1>0$ and $C_2>0$ depending only on the activation function $\rho(\cdot)$.
\end{lemma}

\begin{proof}
For simplicity, let $\|f\|_2:= (\int_{-1}^1|f(x)|^2\d x)^{1/2}$ and $\|f\|_\infty :=\sup_{x\in[-1,1]}|f(x)|$ for a real-valued function $f.$
Let  $h(x):=\ind(x\ge 0)$ and $h_K(x):=\frac{1}{2}\cbr{|Kx|-|Kx-1|}+\frac{1}{2}$. Then we have $\|h_K-h\|_2\le 1/(3K)$.

We now approximate $h_K$ by a DNN. By Lemma A3 (e) of \citep{ohn2019smooth}, there is a DNN $f_{\text{abs}}\in\cF(\ceil{C_3\log K}, 16, K^8\vee C_4, 2)$ for some constants $C_3>0$ and $C_4>0$ such that $\|f_{\text{abs}}(x)-|x|\|_\infty\le 1/K^2$. Then the DNN $f_{\textup{ind}}\in \cF(\ceil{C_3\log K}, 31, K^8\vee C_4, 2)$ defined by
    \begin{equation*}
        f_{\textup{ind}}(x):=\frac{1}{2}\cbr{f_{\text{abs}}(Kx)-f_{\text{abs}}(Kx-1)-1}
    \end{equation*}
satisfies $\|f_{\textup{ind}}(x)-h_K(x)\|_\infty \le 1/K$. Hence the desired result follows from the fact
    \begin{align*}
        \|f_{\textup{ind}}-h\|_2
        &\le \|f_{\textup{ind}}-h_K\|_2+ \|h_K-h\|_2\\
        &\le \|f_{\textup{ind}}-h_K\|_\infty+ \|h_K-h\|_2 \le \frac{4}{K}
    \end{align*}
with $K=4/\epsilon.$
\end{proof}

\begin{theorem}
\label{thm:piece_approx}
Let $\epsilon_0>0$. Let $f^\star\in\cG^{\textsc{piece}}(\alpha, \beta, M, K, R)$. Let $\kappa:=d/\alpha \wedge 2(d-1)/\beta$.  Then there exist positive constants $L_0$, $N_0$, $S_0$, $B_0$, $b_0$ and $F_0$ depending only on $\alpha, \beta, M, K, R$, $\epsilon_0$ and $\rho(\cdot)$ such that, for any $\epsilon\in(0, \epsilon_0)$, there is a DNN
    \begin{equation}
        f\in  \cF^\dnn_\rho\del{L_0\log(1/\epsilon),N_0\epsilon^{-\kappa}, B_0\epsilon^{-1}, F_0, S_0\epsilon^{-\kappa}\log(1/\epsilon)}
    \end{equation}
for a piecewise linear $\rho$ and
    \begin{equation}
        f\in  \cF^\dnn_\rho\del{L_0\log(1/\epsilon),N_0\epsilon^{-\kappa}, B_0\epsilon^{-4(\kappa+2)}, F_0, S_0\epsilon^{-\kappa}\log(1/\epsilon)}
    \end{equation}
for a locally quadratic $\rho,$ 
which satisfies
    \begin{equation}
        \norm{f^\star(\x)-f(\x)}_2\le \epsilon.
    \end{equation}
\end{theorem}
\begin{proof}
For a piecewise linear activation function, Lemma A1 of \cite{ohn2019smooth} and Theorem 1 of \cite{imaizumi2018deep} yield the desired result.

We now focus on locally quadratic activation functions. Let $$f^\star(\x)=\sum_{m=1}^Mg_m^\star(\x)\prod_{k\in[K]}\ind\del{x_{j_{m,k}}\ge h^\star_{m,k}(\x_{-j_{m,k}})}\in\cG^{\textsc{piece}}(\alpha, \beta, M, K, R).$$

By \cref{thm:holder_approx}, there are positive constants $L_{0}$, $N_0$, $S_0$, $B_0$ and $F_0$ such that, for any $g_m^\star>0$ and any $\epsilon>0$ there is a neural network 
    \begin{equation*}
        g_m\in  \cF^\dnn_\rho\del{L_0\log(1/\epsilon),N_0\epsilon^{-d/\alpha}, B_0\epsilon^{-4(d/\alpha+1)}, F_0, S_0\epsilon^{-d/\alpha}\log(1/\epsilon)}
    \end{equation*}
such that $\|g_m-g_m^\star\|_\infty\le \epsilon.$

For the approximation of $\ind\del[0]{x_{j_{m,k}}\ge h^\star_{m,k}(\x_{-j_{m,k}})},$ we combine the results of \cref{thm:holder_approx}, \cref{lem:identity_approx} and \cref{lem:indicator_approx}. Let $h_{m,k}$ be a DNN with depth $L^*_\epsilon=O(\log(1/\epsilon))$ and sparsity $\|\btheta(h_{m,k})\|_0=O(\epsilon^{-2(d-1)/\beta})$ such that $\|h_{m,k}^\star-h_{m,k}\|_\infty\le \epsilon^{2}$ for each $m\in[M]$ and $k\in[K]$. For $L\in\bN$, define $f_{\textup{id}}^{(L)}:=f_{\textup{id}, (L-1)\epsilon_0}\circ f_{\textup{id}, (L-2)\epsilon_0}\circ\cdots\circ f_{\textup{id}, \epsilon_0}\circ f_{\textup{id}, 0}$, where $f_{\textup{id}, \delta}$ is a DNN with $\textsf{depth}(f_{\textup{id}, \delta})=L$ and  $\|\btheta(f_{\textup{id}, \delta})\|_\infty\le  C_1(1+\delta)^2L\epsilon^{-2}$ for some $C_1>0$ satisfying $\sup_{x\in[-\delta, 1+\delta]}|f_{\textup{id}, \delta}(x) -x|\le \epsilon^2/L$. Then
    \begin{align*}
        \sup_{x\in[0, 1]}\abs{f_{\textup{id}}^{(L)}(x)-x}
        &\le \sup_{x\in[0, 1]}\abs{f_{\textup{id}}^{(L)}(x)-f_{\textup{id}}^{(L-1)}(x)}
        +\sup_{x\in[0, 1]}\abs{f_{\textup{id}}^{(L-1)}(x)-x}\\
         &\le \sup_{x\in[-(L-1)\epsilon_0, 1+(L-1)\epsilon_0]}\abs{f_{\textup{id}, (L-1)\epsilon_0}(x)-x}
        +\sup_{x\in[0, 1]}|f_{\textup{id}}^{(L-1)}(x)-x|\\
        &\le L\epsilon^2/L=\epsilon^2.
    \end{align*}
Let  $A:=(1+R+2\epsilon_0^2)$. Define $u_{m,k}^\star$ and $u_{m,k}$ by $u_{m,k}^\star(\x):=A^{-1}(x_{j_{m,k}}- h^\star_{m,k}(\x_{-j_{m,k}}))$ and $u_{m,k}(\x):=A^{-1}(f_{\textup{id}}^{(L^*_\epsilon)}(x_{j_{m,k}})-h_{m,k}(\x_{-j_{m,k}}))$, respectively, so that $\|u_{m,k}^\star\|_\infty\le 1$, $\|u_{m,k}\|_\infty\le 1$ and $\ind\del[0]{x_{j_{m,k}}\ge h^\star_{m,k}(\x_{-j_{m,k}})}=\ind(u_{m,k}^\star(\x)\ge0).$ 
Note that $u_{m,k}$ is the DNN with depth $L_\epsilon^*$ 
and $f_{\textup{ind}}$ approximates the indicator function $\ind(x\ge0)$ by error $\epsilon$ with respect to the $L_2$-norm 
with $\|\btheta(f_{\textup{ind}})\|_\infty=O(\epsilon^{-8})$ by \cref{lem:indicator_approx}, Using these results, we construct the DNN approximating $\ind(u_{m,k}^\star\ge0)$ by $f_{\textup{ind}}\circ u_{m,k}$ as follows. We start with
    \begin{align*}
      & \norm{f_{\textup{ind}}\circ u_{m,k}- \ind(u_{m,k}^\star\ge0)}_2\\
       &\le \norm{f_{\textup{ind}}\circ u_{m,k}- \ind(u_{m,k}\ge0)}_2
       + \norm{\ind(u_{m,k}\ge0)- \ind(u_{m,k}^\star\ge0)}_2.
       \end{align*}
The first term of the right-hand side of the preceding display is bounded by $\epsilon$. For the second term, we have
    \begin{align*}
        &\abs{\ind(u_{m,k}(\x)\ge0)- \ind(u_{m,k}^\star(\x)\ge0)}\\
        &= \ind(x_{j_{m,k}} \ge h_{m,k^\star}(\x_{-j_{m,k}}), f_{\textup{id}}^{(L^*_\epsilon)}(x_{j_{m,k}})<h_{m,k}(\x_{-j_{m,k}}) )\\
        &\quad + \ind(x_{j_{m,k}}<h_{m,k^\star}(\x_{-j_{m,k}}), f_{\textup{id}}^{(L^*_\epsilon)}(x_{j_{m,k}})\ge h_{m,k}(\x_{-j_{m,k}}) )\\
        &\le \ind(x_{j_{m,k}} \ge h_{m,k^\star}(\x_{-j_{m,k}}), x_{j_{m,k}}-\epsilon^2 <h_{m,k}(\x_{-j_{m,k}}) )\\
        &\quad + \ind(x_{j_{m,k}}<h_{m,k^\star}(\x_{-j_{m,k}}), x_{j_{m,k}}+\epsilon^2\ge h_{m,k}(\x_{-j_{m,k}}) ),
    \end{align*}
which implies
    \begin{align*}
        &\int_{0}^1\abs{\ind(u_{m,k}(\x)\ge0)- \ind(u_{m,k}^\star(\x)\ge0)}^2\d x_{j_{m,k}}\\
        &\le \max\cbr{0, \epsilon^2 +h_{m,k}(\x_{-j_{m,k}})-h_{m,k^\star}(\x_{-j_{m,k}}) }\\
        & \quad + \max\cbr{0, h_{m,k^\star}(\x_{-j_{m,k}})-h_{m,k}(\x_{-j_{m,k}})+\epsilon^2 }\\
        &\le 4\epsilon^2.
    \end{align*}
Therefore $\norm{f_{\textup{ind}}\circ u_{m,k}- \ind(u_{m,k}^\star\ge0)}_2\le 3\epsilon$. 

The remaining part of the proof is to approximate the product map. For the map $\z\equiv(z_1,\dots, z_K)\mapsto \prod_{k=1}^Kz_k,$ where $z_k\in [0,1]$, by Lemma A3 (c) of \cite{ohn2019smooth}, there is a DNN $f_{\textup{prod}}\in \cF_\rho^\dnn(C_2\log(1/\epsilon), C_3, C_4\epsilon^{-2}, 1+\epsilon_0)$  for some positive constants $C_2, C_3$ and $C_4$ depending only on $K$ and $\rho(\cdot)$ such that $\sup_{\z\in[0,1]^K}|f_{\textup{prod}}(\z)-\prod_{k=1}^Kz_k|\le \epsilon$. Define $I_{m,k}^\star$ and $I_{m,k}$ by $I^\star_{m,k}(\x):=\ind\del[0]{x_{j_{m,k}}\ge h^\star_{m,k}(\x_{-j_{m,k}})}$ and $I_{m,k}(\x):=f_{\textup{ind}}\circ u_{m,k}(\x)$ and let
    $$\bpsi_{m,k}:=A^{-1}(g_{m}, I_{m,1},  I_{m,2}, \dots,   I_{m,K}).$$
Define the DNN $f_{m}:=A^{K+1}f_{\textup{prod}}\circ \bpsi_{m,k}$. Then we have

    \begin{align*}
         \norm{f_{m}-g_m^\star\prod_{k\in[K]}I^\star_{m,k}}_2
        &=A^{K+1}\norm{f_{\textup{prod}}\circ \bpsi_{m,k}-A^{-(K+1)}g_m^\star\prod_{k\in[K]}I^\star_{m,k}}_2\\
        &\le A^{K+1}\norm{f_{\textup{prod}}\circ \bpsi_{m,k}-A^{-(K+1)}g_{m}\prod_{k\in[K]}I_{m,k}}_2\\
        &\quad +A^{K+1}\norm{(g_{m}/A)\prod_{k\in[K]}(I_{m,k}/A)-(g_m^\star/A)\prod_{k\in[K]}(I^\star_{m,k}/A)}_2\\
        &\le A^{K+1}(\epsilon + (K+1)A^{-1}\epsilon).
    \end{align*}
Since $K$ and $A$ are fixed constants, we get the desired result.
\end{proof}
\end{appendices}

\bibliographystyle{chicago}
\bibliography{reference-deep}

\end{document}

%% file: preamble.tex
\usepackage{amsfonts,amsmath,amssymb,amsthm}
\usepackage{commath,mathtools}
\usepackage{bm, bbm, bbold}
\usepackage{bm,bbm}
\usepackage{upgreek}
\usepackage[title]{appendix}
\usepackage[capitalize, noabbrev]{cleveref}
\usepackage{caption,subcaption}
\usepackage{float}
\newenvironment{equations}{\equation\aligned}{\endaligned\endequation}

\newtheorem{theorem}{Theorem}
\newtheorem{lemma}[theorem]{Lemma} 
\newtheorem{proposition}[theorem]{Proposition}

\newtheorem{definition}[theorem]{Definition}

\DeclareMathOperator*{\argmin}{\arg\!\min}

\newcommand{\ceil}[1]{\left\lceil {#1} \right\rceil}
\newcommand{\floor}[1]{\left\lfloor {#1} \right\rfloor}

\newcommand{\inn}[2]{\left\langle {#1}, {#2} \right\rangle}

\def\btheta{\bm{\uptheta}}

\def\ind{\mathbbm{1}}

\def\one{\mathbf 1}

\def\d{\mathrm{d}}
\def\e{\textup{e}}

\def\E{\mathsf{E}}
\def\P{\mathsf{P}}
\def\R{\mathbb{R}}

\def\dnn{\textsc{dnn}}
\def\clip{\mathsf{clip}}

\makeatletter
\def\greekvectors#1{%
 \@for\next:=#1\do{%
    \def\X##1;{\expandafter\def\csname b##1\endcsname{\bm{\csname##1\endcsname}}}
    \expandafter\X\next;}
 \@for\next:=#1\do{%
    \def\X##1;{\expandafter\def\csname h##1\endcsname{\widehat{\csname##1\endcsname}}}
    \expandafter\X\next;}
 \@for\next:=#1\do{%
    \def\X##1;{\expandafter\def\csname c##1\endcsname{\check{\csname##1\endcsname}}}
    \expandafter\X\next;}
 \@for\next:=#1\do{%
    \def\X##1;{\expandafter\def\csname hb##1\endcsname{\widehat{\bm{\csname##1\endcsname}}}}
    \expandafter\X\next;}
}
\greekvectors{alpha,beta,gamma,delta,epsilon,zeta,kappa,lambda,
    mu,nu,xi,pi,rho,tau,phi,chi,psi,omega,
    Delta,Gamma,Theta,Lambda,Xi,Pi,Sigma,Phi,Psi,Omega}
    
\@tfor\next:=abcfghijlmnopqtwxyzABCDGHIJKLMOQSTUVWXYZ\do{%
    \def\command@factory#1{\expandafter\def\csname #1\endcsname{\mathbf{#1}} }
    \expandafter\command@factory\next
}
\@tfor\next:=abcdefghijklmnpqrstuvwxyzABCDEFGHIJKLMNOPQRSTUVWXYZ\do{%
    \def\command@factory#1{\expandafter\def\csname t#1\endcsname{\widetilde{#1}} }
    \expandafter\command@factory\next
}
\@tfor\next:=abcdefghijklmnopqrstuvwxyzABCDEFGHIJKLMNOPQRSTUVWXYZ\do{%
    \def\command@factory#1{\expandafter\def\csname tb#1\endcsname{\tilde{\mathbf{#1}}} }
    \expandafter\command@factory\next
}
\@tfor\next:=xyzABCDEFGHIJKLMNOPQRSTUVWXYZ\do{%
    \def\command@factory#1{\expandafter\def\csname h#1\endcsname{\widehat{#1}} }
    \expandafter\command@factory\next
}
\@tfor\next:=abcdefghijklmnopqrstuvwxyzABCDEFGHIJKLMNOPQRSTUVWXYZ\do{%
    \def\command@factory#1{\expandafter\def\csname hb#1\endcsname{\widehat{\mathbf{#1}}} }
    \expandafter\command@factory\next
}
\@tfor\next:=abcdeghijklnopqrstuvwxyzABCDEFGHIJKLMNOPQRSTUVWXYZ\do{%
    \def\command@factory#1{\expandafter\def\csname b#1\endcsname{\mathbbm{#1}} }
    \expandafter\command@factory\next
}
\@tfor\next:=ABCDEFGHIJKLMNOPQRSTUVWXYZ\do{%
    \def\command@factory#1{\expandafter\def\csname c#1\endcsname{\mathcal{#1}} }
    \expandafter\command@factory\next
}
\@tfor\next:=abcdefghjklmnopqrstuvwxyzABCDEFGHIJKLMNOPQRSTUVWXYZ\do{
    \def\command@factory#1{\expandafter\def\csname f#1\endcsname{\mathfrak{#1}} }
    \expandafter\command@factory\next
}
\@tfor\next:=fghijklmnopqrstuvwxyzABCDEFGHIJKLMNOPQRSTUVWXYZ\do{%
    \def\command@factory#1{\expandafter\def\csname s#1\endcsname{\mathsf{#1}} }
    \expandafter\command@factory\next
}